\documentclass[11pt]{amsart}
\usepackage{amsopn}
\usepackage{amssymb, amscd}

\newcommand{\nc}{\newcommand}

\nc{\eg}{\mathfrak{e} } \nc{\fg}{\mathfrak{f} } \nc{\vg}{\mathfrak{v} } \nc{\wg}{\mathfrak{w} }
\nc{\zg}{\mathfrak{z} } \nc{\ngo}{\mathfrak{n} } \nc{\kg}{\mathfrak{k} }
\nc{\mg}{\mathfrak{m} } \nc{\bg}{\mathfrak{b} } \nc{\ggo}{\mathfrak{g} }
\nc{\ggob}{\overline{\mathfrak{g}} } \nc{\sog}{\mathfrak{so} }
\nc{\sug}{\mathfrak{su} } \nc{\spg}{\mathfrak{sp} } \nc{\slg}{\mathfrak{sl} }
\nc{\glg}{\mathfrak{gl} } \nc{\cg}{\mathfrak{c} } \nc{\rg}{\mathfrak{r} }
\nc{\hg}{\mathfrak{h} } \nc{\tg}{\mathfrak{t} } \nc{\ug}{\mathfrak{u} }
\nc{\dg}{\mathfrak{d} } \nc{\ag}{\mathfrak{a} } \nc{\pg}{\mathfrak{p} }
\nc{\sg}{\mathfrak{s} } \nc{\affg}{\mathfrak{aff} }

\nc{\pca}{\mathcal{P}} \nc{\nca}{\mathcal{N}} \nc{\lca}{\mathcal{L}}
\nc{\oca}{\mathcal{O}} \nc{\mca}{\mathcal{M}} \nc{\tca}{\mathcal{T}}
\nc{\aca}{\mathcal{A}} \nc{\cca}{\mathcal{C}} \nc{\gca}{\mathcal{G}}
\nc{\sca}{\mathcal{S}} \nc{\hca}{\mathcal{H}} \nc{\bca}{\mathcal{B}}
\nc{\dca}{\mathcal{D}} \nc{\val}{\operatorname{val}}

\nc{\vp}{\varphi} \nc{\ddt}{\tfrac{{\rm d}}{{\rm d}t}} \nc{\dds}{\tfrac{{\rm d}}{{\rm d}s}}
\nc{\dpar}{\tfrac{\partial}{\partial t}} \nc{\im}{\mathtt{i}}

\renewcommand{\Im}{{\rm Im}}

\nc{\SO}{\mathrm{SO}} \nc{\Spe}{\mathrm{Sp}} \nc{\Sl}{\mathrm{SL}}
\nc{\SU}{\mathrm{SU}} \nc{\Or}{\mathrm{O}} \nc{\U}{\mathrm{U}} \nc{\Gl}{\mathrm{GL}}
\nc{\Se}{\mathrm{S}} \nc{\Cl}{\mathrm{Cl}} \nc{\Spein}{\mathrm{Spin}}
\nc{\Pin}{\mathrm{Pin}} \nc{\G}{\mathrm{GL}_n(\RR)} \nc{\g}{\mathfrak{gl}_n(\RR)}

\nc{\RR}{{\Bbb R}} \nc{\HH}{{\Bbb H}} \nc{\CC}{{\Bbb C}} \nc{\ZZ}{{\Bbb Z}}
\nc{\FF}{{\Bbb F}} \nc{\NN}{{\Bbb N}} \nc{\QQ}{{\Bbb Q}} \nc{\PP}{{\Bbb P}}

\nc{\vs}{\vspace{.2cm}} \nc{\vsp}{\vspace{1cm}} \nc{\ip}{\langle\cdot,\cdot\rangle}
\nc{\ipp}{(\cdot,\cdot)} \nc{\la}{\langle} \nc{\ra}{\rangle} \nc{\unm}{\tfrac{1}{2}}
\nc{\unc}{\tfrac{1}{4}} \nc{\und}{\tfrac{1}{16}} \nc{\no}{\vs\noindent}
\nc{\lamkn}{\Lambda^2(\RR^{q+n})^*\otimes\RR^{q+n}} \nc{\lamn}{\Lambda^2(\RR^n)^*\otimes\RR^n} \nc{\lamp}{\Lambda^2\pg^*\otimes\pg}
\nc{\lamg}{\Lambda^2\ggo^*\otimes\ggo} \nc{\lamngo}{\Lambda^2\ngo^*\otimes\ngo}
\nc{\tangz}{{\rm T}^{\rm Zar}} \nc{\mum}{/\!\!/} \nc{\kir}{/\!\!/\!\!/}
\nc{\Ri}{\tfrac{4\Ric_{\mu}}{||\mu||^2}} \nc{\ds}{\displaystyle}
\nc{\ben}{\begin{enumerate}} \nc{\een}{\end{enumerate}} \nc{\f}{\frac}
\nc{\lb}{[\cdot,\cdot]} \nc{\isn}{\tfrac{1}{||v||^2}}
\nc{\gkp}{(\ggo=\kg\oplus\pg,\ip)} \nc{\ukh}{(\ug=\kg\oplus\hg,\ip)}

\nc{\Hess}{\operatorname{Hess}} \nc{\ad}{\operatorname{ad}}
\nc{\Ad}{\operatorname{Ad}} \nc{\rank}{\operatorname{rank}}
\nc{\Irr}{\operatorname{Irr}} \nc{\End}{\operatorname{End}}
\nc{\Aut}{\operatorname{Aut}} \nc{\Inn}{\operatorname{Inn}}
\nc{\Der}{\operatorname{Der}} \nc{\Ker}{\operatorname{Ker}}
\nc{\Iso}{\operatorname{I}} \nc{\Diff}{\operatorname{Diff}}
\nc{\Lie}{\operatorname{Lie}} \nc{\tr}{\operatorname{tr}} \nc{\dif}{\operatorname{d}}
\nc{\sen}{\operatorname{sen}} \nc{\modu}{\operatorname{mod}}
\nc{\Riem}{\operatorname{Rm}} \nc{\Ricci}{\operatorname{Ric}}
\nc{\sym}{\operatorname{sym}} \nc{\symac}{\operatorname{sym^{ac}}}
\nc{\symc}{\operatorname{sym^{c}}} \nc{\scalar}{\operatorname{R}}
\nc{\grad}{\operatorname{grad}} \nc{\ricci}{\operatorname{Rc}}
\nc{\nr}{\operatorname{nr}} \nc{\riccic}{\operatorname{ric^{c}}}
\nc{\riccig}{\operatorname{ric^{\gamma}}} \nc{\Rin}{\operatorname{M}}
\nc{\Le}{\operatorname{L}} \nc{\tang}{\operatorname{T}}
\nc{\level}{\operatorname{level}} \nc{\rad}{\operatorname{r}}
\nc{\abel}{\operatorname{ab}} \nc{\CH}{\operatorname{CH}}
\nc{\mcc}{\operatorname{mcc}} \nc{\Adj}{\operatorname{Adj}}
\nc{\Order}{\operatorname{O}} \nc{\mm}{\operatorname{M}}
\nc{\inj}{\operatorname{inj}} \nc{\proy}{\operatorname{pr}}

\theoremstyle{plain}
\newtheorem{theorem}{Theorem}[section]
\newtheorem{proposition}[theorem]{Proposition}
\newtheorem{corollary}[theorem]{Corollary}
\newtheorem{lemma}[theorem]{Lemma}

\theoremstyle{definition}
\newtheorem{definition}[theorem]{Definition}

\theoremstyle{remark}
\newtheorem{remark}[theorem]{Remark}

\newtheorem{example}[theorem]{Example}

\title{Convergence of homogeneous manifolds}

\author{Jorge Lauret}

\address{FaMAF and CIEM, Universidad Nacional de C\'ordoba, C\'ordoba, Argentina}
\email{lauret@famaf.unc.edu.ar}

\thanks{This research was partially supported by grants from CONICET (Argentina)
and SeCyT (Universidad Nacional de C\'ordoba)}

\begin{document}

\maketitle

\begin{abstract}
We study in this paper three natural notions of convergence of homogeneous manifolds, namely infinitesimal, local and pointed, and their relationship with a fourth one, which only takes into account the underlying algebraic structure of the homogeneous manifold and is indeed much more tractable.  Along the way, we introduce a subset of the variety of Lie algebras which parameterizes the space of all $n$-dimensional simply connected homogeneous spaces with $q$-dimensional isotropy, providing a framework which is very advantageous to approach variational problems for curvature functionals as well as geometric evolution equations on homogeneous manifolds.
\end{abstract}

\tableofcontents


\section{Introduction}

It is often complicated to write rigorous proofs in convergence theory of Riemannian manifolds.  In the homogeneous case, however, it is natural to expect that an `algebraic' notion of convergence may help.  With this aim in mind, we study in this paper three natural notions of convergence of homogeneous manifolds, namely infinitesimal, local and pointed, and their relationship with a fourth one, which only takes into account the underlying algebraic structure of the homogeneous manifold and is indeed much more tractable.  Along the way, we introduce a set $\hca_{q,n}$ of $(q+n)$-dimensional Lie algebras which parameterizes the space of all $n$-dimensional simply connected homogeneous spaces with $q$-dimensional isotropy, providing a framework which is very advantageous to approach variational problems for curvature functionals as well as geometric evolution equations on homogeneous manifolds.

\subsection{Convergence} In order to define convergence of a sequence $(M_k,g_k)$ of homogeneous manifolds to a homogeneous manifold $(M,g)$, it is customary to start by requiring the existence of a sequence $\Omega_k\subset M$ of open neighborhoods of a basepoint $p\in M$ together with embeddings $\phi_k:\Omega_k\longrightarrow M_k$ such that $\phi_k^*g_k\to g$ smoothly as $k\to\infty$ (in particular, all manifolds are of a given dimension $n$).  The size of the $\Omega_k$'s will make the difference, and according to some possible behaviors, one gets the following notions of convergence in increasing degree of strength:

\begin{itemize}
\item {\it infinitesimal}: no condition on $\Omega_k$, it may even happen that $\bigcap\Omega_k=\{ p\}$ (i.e. only the germs of the metrics at $p$ are involved).  Nevertheless, the sequence $(M_k,g_k)$ has necessarily bounded geometry by homogeneity.

\item {\it local}: there is a nonempty open subset $\Omega\subset\Omega_k$ for all sufficiently large $k$.  A positive lower bound for the injectivity radii therefore holds.

\item {\it pointed or Cheeger-Gromov}: $\Omega_k$ eventually contains any compact subset of $M$.  Here topology issues come in.
\end{itemize}

On the other hand, we know that homogeneous manifolds have a prominent `algebraic' side, and the point is to what extent this is related to the above notions of convergence.  Let us consider for each $(M_k,g_k)$ a presentation $G_k/K_k$ as a homogeneous space endowed with a $G_k$-invariant metric $g_k$.  By just requiring $\dim{G_k}$ to be constant on $k$, we can assume that a fixed vector space decomposition $\ggo=\kg\oplus\pg$ gives the reductive decomposition for all $G_k/K_k$ and a fixed inner product $\ip$ on $\pg$ is the value of $g_k$ at the origin $eK_k$ for all $k$.
In this way, the only algebraic data which vary are the sequence of Lie brackets $\mu_k$ on the vector space $\ggo$ such that $(\ggo,\mu_k)$ is the Lie algebra of $G_k$.  Thus, a fourth notion of convergence for homogeneous manifolds comes into play: the standard convergence of brackets $\mu_k\to\lambda$ as vectors in $\Lambda^2\ggo^*\otimes\ggo$, where $\lambda$ is the corresponding Lie bracket for the limit $(M,g)=G/K$.

Our main results can be described as follows.  We show that the convergence of brackets $\mu_k\to\lambda$ is essentially equivalent to the infinitesimal convergence of homogeneous manifolds $(G_k/K_k,g_k)\to (G/K,g)$ (cf. \ref{convmu2}), and secondly, that in order to get the stronger local convergence, it is  sufficient to have a positive lower bound for the Lie injectivity radii of the sequence (cf. \ref{convmu}).  The {\it Lie injectivity radius} of a homogeneous space $(G/K,g)$ is the largest $r>0$ such that its canonical coordinates $\pi\circ\exp$ are defined on the euclidean ball of radius $r$ in $(\pg,\ip)$ (cf. \ref{lieinj}).  Notice that local convergence implies pointed subconvergence of $(M_k,g_k)$ to a homogeneous manifold locally isometric to $(M,g)$, by the compactness theorem (cf. \ref{pcomp}), but we show that such limits may topologically vary for different subsequences.

It is important to note that for left-invariant metrics on Lie groups (i.e. $K_k=\{ e\}$ for all k), the positive lower bound for the Lie injectivity radii follows at once from the convergence $\mu_k\to\lambda$ (cf. \ref{lieinjLG}), giving rise to stronger results in this case (cf. \ref{convmu4}).

\subsection{The space of homogeneous manifolds}  Recall that the data $(\ggo=\kg\oplus\pg,\ip)$ of a homogeneous space can be canonically fixed at the level of inner product vector spaces.  This motivates us to consider the set $\hca_{q,n}\subset \Lambda^2\ggo^*\otimes\ggo$, where $q:=\dim{\kg}$, $n:=\dim{\pg}$, of those Lie brackets satisfying the technical conditions (cf. (h1)-(h4) in Section \ref{varhs}) which allow us to define a simply connected homogeneous space $(G_\mu/K_\mu,g_\mu)$ attached to each $\mu\in\hca_{q,n}$ with $\Lie(G_\mu)=(\ggo,\mu)$, $\Lie(K_\mu)=(\kg,\mu|_{\kg\times\kg})$ and $g_\mu(eK_\mu)=\ip$.  The set $\hca_{q,n}$ therefore parameterizes the space of all $n$-dimensional simply connected homogeneous spaces with $q$-dimensional isotropy.

This approach, that varies Lie brackets rather than metrics has been used for decades, though only in the case of left-invariant metrics on Lie groups (i.e. $q=0$).  We mention, among many others, just a few instances.  It was used in  \cite{Mln,Hnt,Ebr} to study curvature properties of Lie groups; in the structure results for Einstein solvmanifolds obtained in \cite{Hbr,standard}; in viewing nilsolitons as critical points and their classification (cf. \cite{soliton,einsteinsolv,Nkl,Wll,Jbl}); in the study of the Ricci flow for $3$-dimensional homogeneous geometries (cf. \cite{GlcPyn}) and for nilmanifolds (cf. \cite{Gzh,Pyn,nilricciflow}).  The approach can also be applied to complex and symplectic homogeneous geometry (cf. \ref{gs}).  In many of these articles, an intriguing relationship with the geometric invariant theory of the variety of Lie algebras, including  closed orbits, categorical quotients, moment maps and Kirwan stratification, has been exploited in one way or another.

\subsection{Examples} The paper includes plenty of situations which illustrate our approach and provide examples and counterexamples to some of the speculations one might make on convergence issues.  Namely:
\begin{itemize}
\item Subsets of $\hca_{0,3}$ and $\hca_{1,3}$ reaching all $3$-dimensional geometries (cf. \ref{ex0-3}, \ref{ex1-3}).

\item A family in $\hca_{1,5}$ parameterizing all homogeneous metrics on $S^3\times S^2$ (cf. \ref{ex1-5}).

\item A $6$-parameter family in $\hca_{1,7}$ attaining any $\SU(3)$-invariant metric on all (generic) Aloff-Wallach spaces (cf. \ref{AW}).

\item A sequence of Aloff-Wallach spaces which infinitesimally converges to another Aloff-Wallach space, but such that it does not admit any pointed or local convergent subsequence (cf. \ref{AW3}).

\item A sequence of alternating left-invariant metrics on $S^3$ (Berger spheres) and $\widetilde{\Sl_2}(\RR)$ which locally converges to a flat metric on the solvable Lie group $E(2)$, but the corresponding subsequences pointed converges to $S^1\times\RR^2$ and $\RR^3$, respectively (cf. \ref{berger}).

\item A divergent sequence $\mu_k\in\hca_{0,3}$ of left-invariant metrics on $\widetilde{\Sl_2}(\RR)$ which nevertheless pointed converges to $\RR\times H^2$, where $H^2$ denotes the $2$-dimensional hyperbolic space.  $\mu_k$ is actually isometric to a convergent sequence in $\hca_{1,3}$ (cf. \ref{hyp}).

\item A sequence $\mu_k\in\hca_{1,5}$ of homogeneous metrics on $S^3\times S^2$ converging to a Lie bracket $\lambda$ which is not in $\hca_{1,5}$.  However, $\lambda$ can be viewed as an element of $\hca_{2,4}$, giving rise to a collapsing of the $\mu_k$ with bounded curvature to a metric on $S^2\times S^2$ (cf. \ref{coll}).
\end{itemize}

\subsection{Ricci flow} Our true motivation to study the `algebraic' convergence of homogeneous manifolds is that the Ricci flow $g(t)$ starting at a homogeneous manifold $(M,g_0)$ is proved in \cite{homRF} to be equivalent to an evolution equation for Lie brackets in the following precise sense: if $(M,g_0)=(G_{\mu_0}/K_{\mu_0},g_{\mu_0})$, $\mu_0\in\hca_{q,n}$, then the solution $\mu=\mu(t)$ to the so called {\it bracket flow} given by the ODE
$$
\ddt\mu=\mu\left(\left[\begin{smallmatrix} 0&0\\ 0&\Ricci_{\mu}
\end{smallmatrix}\right]\cdot,\cdot\right) + \mu\left(\cdot,\left[\begin{smallmatrix} 0&0\\ 0&\Ricci_{\mu}
\end{smallmatrix}\right]\cdot\right) -\left[\begin{smallmatrix} 0&0\\ 0&\Ricci_{\mu}
\end{smallmatrix}\right]\mu(\cdot,\cdot), \qquad \mu(0)=\mu_0,
$$
where $\Ricci_\mu:\pg\longrightarrow\pg$ denotes the Ricci operator of $g_\mu$ at the origin, stays in $\hca_{q,n}$ for all $t$ and
$$
g(t)=\vp(t)^*g_{\mu(t)}
$$
for some family $\vp(t):M=G_{\mu_0}/K_{\mu_0}\longrightarrow G_{\mu(t)}/K_{\mu(t)}$ of time-dependent equivariant diffeomorphisms.  The fixed points of any normalized bracket flow $c(t)\mu(\tau(t))$ are Ricci solitons, and the solutions $g(t)$ and $\mu(t)$ have identical maximal interval of existence time and curvature behavior.  Moreover, as there always exists a convergent subsequence $\mu_k:=\tfrac{1}{\|\mu(t_k)\|}.\mu(t_k)\to\lambda$, one can apply the convergence results obtained in this paper to get pointed subconvergence of the Ricci flow $g(t)$ (up to scaling) to a Ricci soliton $g_\lambda$ (usually nonflat), provided $\lambda\in\hca_{q,n}$ and there is a lower bound for the Lie injectivity radii $r_{\mu_k}$.

\vs \noindent {\it Acknowledgements.}  The author would like to thank John Lott and Peter Petersen for helpful comments, and also to the referee for his/her invaluable corrections and suggestions on a first version of this paper.

\section{Classical setting}\label{hm}

A Riemannian manifold $(M,g)$ is said to be {\it homogeneous} if its isometry group
$\Iso(M,g)$ acts transitively on $M$.  $\Iso(M,g)$ is known to be naturally a Lie
group such that its action on $M$ is smooth and the isotropy subgroup $\Iso_p(M,g)$ at every point $p\in
M$ is compact.  A {\it homogeneous Riemannian space} is instead a differentiable manifold
$G/K$, where $G$ is a Lie group and $K\subset G$ a closed Lie subgroup, endowed with
a $G$-invariant Riemannian metric. Both concepts are of course intimately related,
though not in a one-to-one way. When studying a geometric problem on
homogeneous manifolds, it is often very useful and healthy to capture the
relevant algebraic information and present the hypotheses and the problem in
`algebraic' terms.  We refer to the books \cite[Chapter X]{KbyNmz} and \cite[Chapter 7]{Bss} for a more detailed treatment of what follows.

Let $(M,g)$ be a connected homogeneous manifold.  Then each closed Lie subgroup $G\subset\Iso(M,g)$ acting transitively on $M$ (which can be assumed to be connected) gives rise to a
presentation of $(M,g)$ as a homogeneous space $(G/K,g_{\ip})$, where
$K=G\cap\Iso_p(M,g)$ for some $p\in M$.  Since $K$ turns out to be compact, there always
exists an $\Ad(K)$-invariant direct sum decomposition
$$
\ggo=\kg\oplus\pg,
$$
where $\ggo$ and $\kg$ are respectively the Lie algebras of $G$ and $K$.  Such a decomposition is called {\it reductive} and is not necessarily unique.  Thus $\pg$ can be naturally identified with the tangent space
$$
\pg\equiv\tang_pM=\tang_{eK}G/K,
$$
by taking the value at $p$ of the Killing vector fields
corresponding to elements of $\pg$.  We denote by $g_{\ip}$ the $G$-invariant metric on $G/K$ determined by
$$
\ip:=g(p),
$$
the $\Ad(K)$-invariant inner
product on $\pg$ defined by $g$.

Any kind of curvature of $(G/K,g_{\ip})$, and hence
of $(M,g)$, can therefore be computed in terms of the inner product vector space
$(\pg,\ip)$ and the Lie bracket $\lb$ of $\ggo$ (see for instance \cite[Chapter 7]{Bss}).

\begin{remark}
A homogeneous space $(G/K,g_{\ip})$ will always be assumed to carry a fixed
$\Ad(K)$-invariant decomposition $\ggo=\kg\oplus\pg$.
\end{remark}

In order to get a presentation $(M,g)=(G/K,g_{\ip})$ of a connected homogeneous manifold as a homogeneous space, there is no need for $G\subset\Iso(M,g)$ to hold, that is, an {\it effective} action.  It is actually enough to have a transitive action of $G$ on $M=G/K$, where $K$ is the isotropy subgroup at some point, which is {\it almost-effective} (i.e. $K$ contains no non-discrete normal subgroup of
$G$, or equivalently, the normal subgroup $\{ g\in G:ghK=hK, \;\forall h\in G\}$ is
discrete), along with a decomposition $\ggo=\kg\oplus\pg$ and an inner product $\ip$ on $\pg$, both of them $\Ad(K)$-invariant.  In particular, $G$ can always be chosen to be {\it simply connected} (i.e. connected and with trivial fundamental group) and almost-effective.  If in addition $M$ is simply connected, then $K$ must be connected (although not necessarily compact); and conversely, if $G$ is simply connected and $K$ connected, then $M$ is simply connected (use the homotopy sequence of the fibration $G\longrightarrow G/K$).

The set of all $G$-invariant metrics on $G/K$ is in one-to-one correspondence with the set of all $\Ad(K)$-invariant inner products on $\pg$.  Such a set can be naturally identified with a symmetric subspace (possibly flat) of the symmetric space $\Gl_n^+(\RR)/\SO(n)$ and so it is diffeomorphic to a euclidean space.  It could however be far from covering all homogeneous metrics on the manifold $G/K$.

\section{Varying Lie brackets viewpoint}\label{varhs}

A simply connected homogeneous space $(G/K,g_{\ip})$ with $G$ simply connected is completely characterized (as $K$ must be connected) by the following `algebraic' data:
\begin{itemize}
\item[ ] the vector space decomposition $\ggo=\kg\oplus\pg$;

\item[ ] the inner product $\ip$ on $\pg$;

\item[ ] the Lie bracket $\lb$ of $\ggo$.
\end{itemize}
As the pair $(\ggo=\kg\oplus\pg,\ip)$ can be canonically fixed, this suggests varying Lie brackets to cover a large number of homogeneous manifolds at the same space. In this light, we shall define in this section a set $\hca_{q,n}$ whose elements are simply connected homogeneous spaces and such that any simply connected homogeneous space $(G/K,g_{\ip})$ of dimension $n$ and $\dim{K}=q$ is isometric to at least one point in $\hca_{q,n}$.

Let us fix a decomposition
$$
\RR^{q+n}=\RR^q\oplus\RR^n,
$$
together with the canonical inner product $\ip$ on $\RR^n$.  We consider the space of all skew-symmetric algebras (or brackets) of dimension $q+n$, which is
parameterized by the vector space
$$
V_{q+n}:=\{\mu:\RR^{q+n}\times\RR^{q+n}\longrightarrow\RR^{q+n} : \mu\; \mbox{bilinear and
skew-symmetric}\}.
$$
For any $x\in\RR^{q+n}$, we denote left multiplication (or adjoint action) as usual by $\ad_{\mu}{x}(y)=\mu(x,y)$ for all $y\in\RR^{q+n}$.

A homogeneous space can be associated to an element $\mu\in V_{q+n}$ provided the following conditions hold for $\mu$:

\begin{itemize}
\item [(h1)]  $\mu$ satisfies the Jacobi condition, $\mu(\RR^q,\RR^q)\subset\RR^q$ and $\mu(\RR^q,\RR^n)\subset\RR^n$.

\item[(h2)] If $G_\mu$ denotes the simply connected Lie group with Lie algebra $(\RR^{q+n},\mu)$ and $K_\mu$ is the connected Lie subgroup of $G_\mu$ with Lie algebra $\RR^q$, then $K_\mu$ is closed in $G_\mu$.

\item[(h3)] $\ip$ is $\ad_{\mu}{\RR^q}$-invariant (i.e. $(\ad_{\mu}{z}|_{\RR^n})^t=-\ad_{\mu}{z}|_{\RR^n}$ for all $z\in\RR^q$).

\item[(h4)] $\{ z\in\RR^q:\mu(z,\RR^n)=0\}=0$.
\end{itemize}

Indeed, by (h2), the simply connected topological space $G_{\mu}/K_{\mu}$ admits a unique differentiable manifold structure such that the quotient map $\pi_\mu:G_\mu\longrightarrow G_\mu/K_\mu$ is smooth and admits local smooth sections, or equivalently, the $G_{\mu}$-action on
$G_{\mu}/K_{\mu}$ is smooth (see \cite[3.58,3.63]{Wrn}).  Such an action is almost-effective by (h4), and it follows from (h3) that $\ip$ is
$\Ad(K_{\mu})$-invariant as $K_{\mu}$ is connected.  All this is already
enough to get a homogeneous space,
\begin{equation}\label{hsmu}
\mu\in\hca_{q,n}\rightsquigarrow\left(G_{\mu}/K_{\mu},g_\mu\right),
\end{equation}
with $\Ad(K_{\mu})$-invariant decomposition $\RR^{q+n}=\RR^q\oplus\RR^n$ and $g_\mu(eK_\mu)=\ip$ (see
\cite[p.200]{KbyNmz} or \cite[7.24,7.12]{Bss}), where
\begin{equation}\label{hkn}
\hca_{q,n}:=\{\mu\in V_{q+n}: \;\mbox{conditions (h1)-(h4) hold for}\;\mu\}.
\end{equation}
If for $u\in G_\mu$ we denote by $\tau_\mu(u):G_\mu/K_\mu\longrightarrow G_\mu/K_\mu$ the diffeomorphism
$$
\tau_\mu(u)(vK_\mu):=uvK_\mu, \qquad v\in G_\mu,
$$
then the metric $g_\mu$ is given by
\begin{equation}\label{gmu}
g_\mu(uK_\mu)\left(\dif\tau_\mu(u)|_{eK_\mu}x, \dif\tau_\mu(u)|_{eK_\mu}y\right)=\la x,y\ra, \qquad\forall x,y\in\RR^n, \quad u\in G_\mu.
\end{equation}

We note that any $n$-dimensional homogeneous Riemannian space $(G/K,g_{\ip})$ with $G$ simply connected and $\Ad(K)$-invariant decomposition $\ggo=\kg\oplus\pg$ which is almost-effective can be identified with some $\mu\in\hca_{q,n}$, where $q=\dim{K}$.  Indeed, one just has to fix a basis of $\kg$ and an orthonormal basis of $\pg$ in order to get identifications $\kg=\RR^q$, $\pg=\RR^n$, and so $\mu$ is precisely the Lie bracket of $\ggo$.  In particular, in the set
$$
\hca_n:=\bigcup\limits_{q=0}^{n(n-1)/2}\,\hca_{q,n},
$$
all simply connected homogeneous Riemannian manifolds of dimension $n$ (up to isometry) are represented, though often by several different points which may even represent inequivalent homogeneous spaces (see Section \ref{equivtypes}).

If $\hca_{q,n}$ is nonempty, which is not always the case (e.g. $\hca_{2,3}=\emptyset$), then there must be a flat element in $\hca_{q,n}$.  Indeed, for any $\mu\in\hca_{q,n}$ one can define $\lambda\in V_{q+n}$ by $\lambda|_{\RR^q\times\RR^{q+n}}:=\mu$, $\lambda|_{\RR^n\times\RR^{n}}:=0$, for which conditions (h1)-(h4) can be easily verified, getting the flat manifold $\left(G_{\lambda}/K_{\lambda},g_\lambda\right)=((K\ltimes\RR^n)/K,g_{\ip})$ for some compact subgroup $K\subset\Or(n)$.

Concerning the question of what kind of subset of $V_{q+n}$ the space $\hca_{q,n}$ is, we note that conditions (h1) and (h3) are closed, they are even defined by polynomial equations on $\mu$.  On the contrary, (h4) is open and (h2) may impose a very subtle condition on $\mu$, as Examples \ref{ex1-5}, \ref{AW} show.  Notice that $\hca_{q,n}$ is a cone, i.e. invariant by any nonzero scaling.

\begin{example}\label{ex0-n}
If $q=0$, then conditions (h2)-(h4) trivially hold and (h1) is just the Jacobi condition for $\mu$.  Thus $\hca_{0,n}=\lca_n$, the variety of $n$-dimensional Lie algebras, and the set $\{\left(G_\mu,g_\mu\right):\mu\in\lca_n\}$ parameterizes the set of all left-invariant metrics on simply connected Lie groups of dimension $n$ (cf. Section \ref{lgcase} for a more detailed study of this case).
\end{example}

The next two examples reach all $3$-dimensional geometries.

\begin{example}\label{ex0-3}
Let $\mu=\mu_{a,b,c}$ be the Lie bracket in $\hca_{0,3}=\lca_3$ defined by
$$
\mu(e_2,e_3)=ae_1, \qquad \mu(e_3,e_1)=be_2, \qquad \mu(e_1,e_2)=ce_3.
$$
Their isomorphism classes are invariant by permutation of $(a,b,c)$ and scaling, so we can assume $a\geq b\geq c$ and that at most one of them is negative.  The Lie algebras (and geometries) attained by this family are
\begin{equation}\label{ex0-32}
\mu\simeq\left\{\begin{array}{lcl}
\sug(2), && a,b,c>0; \\
\slg_2(\RR), && a,b>0,\, c<0; \\
\eg(2), && a,b>0,\, c=0; \\
\eg(1,1), && a>0,\, b=0,\, c<0; \\
\hg_3, && a>0,\, b=c=0; \\
\RR^3, && a=b=c=0;
\end{array}\right. \quad
G_\mu=\left\{\begin{array}{l}
S^3; \\
\widetilde{\Sl_2}(\RR); \\
E(2); \\
Sol; \\
Nil; \\
\RR^3;
\end{array}\right.
\end{equation}
where $\eg(2)$, $\eg(1,1)$ are unimodular solvable Lie algebras and $\hg_3$ is the $3$-dimensional Heisenberg Lie algebra.  These are all $3$-dimensional unimodular real Lie algebras, and any left-invariant metric on any of the corresponding simply connected Lie groups is isometric to some $\mu_{a,b,c}$ (see \cite[Section 4]{Mln}).  We have added on the right of (\ref{ex0-32}) the $3$-dimensional geometries from the Geometrization Conjecture which are covered by the family $\mu_{a,b,c}$ by using the standard notation.  With the only exception of $S^3$, they are all diffeomorphic to the euclidean space $\RR^3$.  In \cite{GlcPyn}, this presentation as a space of Lie brackets is used to study the Ricci flow of these metrics.
\end{example}

\begin{example}\label{ex1-3}
Consider the decomposition $\RR^4=\RR\oplus\RR^3$ and the bracket $\mu=\mu_{a,b,c,d}\in V_{1+3}$ given by
$$
\left\{\begin{array}{lll}
\mu(e_3,e_0)=de_2, & \mu(e_2,e_3)=ae_1+be_0, & \mu(e_3,e_1)=ce_2, \\
\mu(e_0,e_2)=de_3, &                           & \mu(e_1,e_2)=ce_3.
\end{array}\right.
$$
It is straightforward to see that conditions (h1) and (h3) hold and that (h4) does if and only if $d\ne 0$.  By computing the Killing form, it is easy to conclude that the Lie algebras (and geometries) attained by this family are
\begin{equation}\label{ex1-32}
\mu\simeq\left\{\begin{array}{lcl}
\RR\oplus\sug(2), && ac+bd>0; \\
\RR\oplus\slg_2(\RR), && ac+bd<0; \\
\RR\ltimes \eg(2),\; \RR\ltimes\hg_3, && ac+bd=0;
\end{array}\right. \quad
G_\mu/K_\mu=\left\{\begin{array}{l}
S^3, \; \RR\times S^2; \\
\widetilde{\Sl_2}(\RR),\; \RR\times H^2; \\
E(2),\; Nil,\; \RR^3.
\end{array}\right.
\end{equation}
In the case when $ac+bd>0$, one can use the isomorphism $G_\mu\simeq\RR\times\SU(2)$ to see that $K_\mu$ is a spiral inside a cylinder $\RR\times S^1$ and thus $K_\mu$ is closed in $G_\mu$.  Notice that otherwise, any Lie subgroup of $G_\mu$ is closed, so that condition (h2) is always satisfied.  We conclude that $\mu_{a,b,c,d}\in\hca_{1,3}$ if and only if $d\ne 0$.  For $a=0$ and $b\ne 0$ we obtain the geometries $\RR\times S^2$ and $\RR\times H^2$, where $H^2$ denotes the $2$-dimensional hyperbolic space.  All the remaining homogeneous metrics $g_{\mu_{a,b,c,d}}$ can be alternatively viewed as left-invariant metrics on $3$-dimensional unimodular Lie groups with an extra symmetry, and hence they have all already appeared in Example \ref{ex0-3}.
\end{example}

All homogeneous metrics on $S^3\times S^2$ can be attained as follows.

\begin{example}\label{ex1-5}
Consider the decomposition $\RR^6=\RR\oplus\RR^5$ and the bracket $\mu=\mu_{p,q,a,b,c,d,e,f}\in V_{1+5}$ given by
$$
\left\{\begin{array}{lll}
\mu(e_0,e_2)=pe_3, & \mu(e_1,e_2)=ee_3, & \mu(e_2,e_3)=ae_0+be_1, \\
\mu(e_0,e_3)=-pe_2, &  \mu(e_1,e_3)=-ee_2,  & \mu(e_4,e_5)=ce_0+de_1. \\
\mu(e_0,e_4)=qe_5, & \mu(e_1,e_4)=fe_5, &  \\
\mu(e_0,e_5)=-qe_4, &  \mu(e_1,e_5)=-fe_4,  &
\end{array}\right.
$$
It is easy to see that the conditions to get $\mu\in\hca_{1,5}$ can be written as follows:
\begin{itemize}
\item[(h1)] $aq+bf=0$,  $\quad cp+de=0$.
\item[(h2)] $p/q\in\QQ$.
\item[(h3)] always holds.
\item[(h4)] $(p,q)\ne(0,0)$.
\end{itemize}
If we assume that $pf-qe\ne 0$, then some of the Lie algebras involved are
$$
\mu\simeq\left\{\begin{array}{lcl}
\sug(2)\oplus\sug(2), && ap+be>0,\quad cq+df>0; \\
\slg_2(\RR)\oplus\slg_2(\RR), && ap+be<0,\quad cq+df<0; \\
\sug(2)\oplus\slg_2(\RR), && (ap+be)(cq+df)<0,
\end{array}\right.
$$
which can be viewed as Lie algebras of matrices in the following way:
$$
\begin{array}{lll}
e_0=\unm(pX_1,qX_1), & e_2=\unm(rX_2,0), & e_4=\unm(0,sX_2), \\ \\
e_1=\unm(eX_1,fX_1), & e_3=\unm(rX_3,0), & e_5=\unm(0,sX_3),
\end{array}
$$
where $r=|ap+be|^{\unm}$, $s=|cq+df|^{\unm}$ and $\{ X_1,X_2,X_3\}\subset\glg_2(\CC)$ is a basis of either $\sug(2)$ or $\slg_2(\RR)$ such that
$$
[X_1,X_2]=2X_3, \quad [X_1,X_3]=-2X_2, \quad [X_2,X_3]=\pm 2X_1.
$$
The equivalence between condition (h2) and $p/q\in\QQ$ is now more transparent, as $K_\mu=e^{\RR e_0}=\{ (e^{tpX_1},e^{tqX_1}):t\in\RR\}$ and one may take $X_1=\left[\begin{smallmatrix} i&0\\ 0&-i\end{smallmatrix}\right]$ for $\sug(2)$ and $X_1=\left[\begin{smallmatrix} 0&-1\\ 1&0\end{smallmatrix}\right]$ for $\slg_2(\RR)$.  A particularly interesting case is when $\mu\simeq\sug(2)\oplus\sug(2)$, since the homogeneous spaces $G_\mu/K_\mu=(\SU(2)\times\SU(2))/S^1$ are all diffeomorphic to $S^3\times S^2$, and actually any homogeneous metric on $S^3\times S^2$ is represented in $\hca_{1,5}$ by a tuple $(p,...,f)$ (cf. e.g. \cite[Example 6.8]{BhmWngZll}).  For different values of $p,...,f$ one gets many other homogeneous spaces, including left-invariant metrics on solvmanifolds as $E(2)\times\RR^2$ and on nilmanifolds as $Nil\times\RR^2$ or the $5$-dimensional Heisenberg Lie group $H_5$.
\end{example}

\begin{remark}
Any of the brackets considered in the above example can also be viewed as an element in $\hca_{2,4}$ by putting
$$
\RR^6=\RR^2\oplus\RR^4 = \la e_0,e_1\ra\oplus\la e_2,...,e_5\ra,
$$
which is easily seen to cover all homogeneous metrics on $S^2\times S^2$, $S^2\times H^2$ and $H^2\times H^2$.
\end{remark}

In the following example, we cover all $\SU(3)$-invariant metrics on each (generic) Aloff-Wallach space $\SU(3)/S^1_{p,q}$.

\begin{example}\label{AW} ({\it Aloff-Wallach spaces})
Consider the decomposition $\RR^8=\RR\oplus\RR^7$ and the bracket $\mu=\mu_{p,q,a,b,c,d}\in V_{1+7}$, $a,b,c,d>0$, given by
$$
\left\{\begin{array}{ll}
\mu(e_0,e_2)=-d(p+2q)e_3, & \mu(e_1,e_2)=-pe_3,  \\
\mu(e_0,e_3)=d(p+2q)e_2, &  \mu(e_1,e_3)=pe_2,  \\
\mu(e_0,e_4)=-d(2p+q)e_5, & \mu(e_1,e_4)=qe_5,   \\
\mu(e_0,e_5)=d(2p+q)e_4, &  \mu(e_1,e_5)=-qe_4,   \\
\mu(e_0,e_6)=-d(p-q)e_7, &  \mu(e_1,e_6)=(p+q)e_7,  \\
\mu(e_0,e_7)=d(p-q)e_6, &  \mu(e_1,e_7)=-(p+q)e_6,   \\ \\

\mu(e_4,e_6)=\mu(e_5,e_7)=-(\tfrac{3bcd}{a})^{\unm}e_2, & \mu(e_5,e_6)=-\mu(e_4,e_7)=-(\tfrac{3bcd}{a})^{\unm}e_3, \\

\mu(e_6,e_2)=-\mu(e_7,e_3)=-(\tfrac{3acd}{b})^{\unm}e_4, & \mu(e_7,e_2)=\mu(e_6,e_3)=-(\tfrac{3acd}{b})^{\unm}e_5, \\

\mu(e_2,e_4)=\mu(e_3,e_5)=-(\tfrac{3abd}{c})^{\unm}e_6, & \mu(e_3,e_4)=-\mu(e_2,e_5)=(\tfrac{3abd}{c})^{\unm}e_7, \\ \\

\mu(e_2,e_3)=-a(p+2q)e_0-3adpe_1, & \mu(e_4,e_5)=-b(2p+q)e_0+3bdqe_1, \\
 \mu(e_6,e_7)=-c(p-q)e_0+3cd(p+q)e_1. & \\
\end{array}\right.
$$

We have that $(\RR^8,\mu)$ is always isomorphic to $\sug(3)$, as these are precisely the Lie bracket relations for the basis $\{ e_0,...,e_8\}$ of $\sug(3)$ given by
$$
\begin{array}{lll}
e_0=\left[\begin{smallmatrix} ipd&&\\ &iqd&\\ &&-i(p+q)d\end{smallmatrix}\right], &
e_1=\tfrac{1}{3}\left[\begin{smallmatrix} -i(p+2q)&&\\ & i(2p+q)&\\ && i(q-p)\end{smallmatrix}\right], & \\ \\

e_2=(3ad)^{\unm}\left[\begin{smallmatrix} 0&&\\ &0&-1\\ &1&0\end{smallmatrix}\right], &
e_3=(3ad)^{\unm}\left[\begin{smallmatrix} 0&&\\ &0&i\\ &i&0\end{smallmatrix}\right], &
e_4=(3bd)^{\unm}\left[\begin{smallmatrix} 0&&-1\\ &0&\\ 1&&0\end{smallmatrix}\right], \\ \\

e_5=(3bd)^{\unm}\left[\begin{smallmatrix} 0&&i\\ &0&\\ i&&0\end{smallmatrix}\right], &
e_6=(3cd)^{\unm}\left[\begin{smallmatrix} 0&-1&\\ 1&0&\\ &&0\end{smallmatrix}\right], &
e_7=(3cd)^{\unm}\left[\begin{smallmatrix} 0&i&\\ i&0&\\ &&0\end{smallmatrix}\right]. \\
\end{array}
$$

Conditions (h1), (h3) and (h4) are all satisfied by any of these $\mu$'s, and concering (h2), we note that
$$
S^1_{p,q}:=K_\mu=e^{\RR e_0},
$$
is closed in $\SU(3)$ (and hence $S^1_{p,q}\simeq S^1$) if and only if $p/q\in\QQ$ (think of $S^1_{p,q}$ as a subgroup of the maximal torus $S^1\times S^1$of $\SU(3)$).  As a differentiable manifold, $G_\mu/K_\mu$ only depends on $p,q$, and so we define
$$
W_{p,q}:=G_\mu/K_\mu=\SU(3)/S^1_{p,q}, \qquad\forall \mu=\mu_{p,q,a,b,c,d}.
$$
These homogeneous manifolds are called in the literature {\it Aloff-Wallach} spaces and have been extensively studied (cf. e.g. \cite{AlfWll,KrcStl,Krg,Ptt}).  By fixing $p,q$ and varying $a,b,c,d$ we get all $\SU(3)$-invariant metrics on $W_{p,q}$ if $p\ne \pm q$ (cf. \cite[Corollary 4.3]{Ptt}).  We note for future use that $W_{rp,rq}=W_{p,q}$ as differentiable manifolds for any $r\in\RR$.
\end{example}

\begin{remark}\label{gs}
If instead of $\ip$ we fix a complex structure $J$ on $\RR^n$ (i.e. an endomorphism such that $J^2=-I$) and change condition (h3) by $[\ad_\mu{\RR^q}|_{\RR^n},J]=0$, then what each $\mu\in\hca_{q,n}$ will represent is a homogeneous space endowed with a left-invariant almost-complex structure.  By adding the integrability of $(G_\mu/K_\mu,J)$ as condition (h5) in the definition of $\hca_{q,n}$, which happens to only depend on $\mu$ and in a polynomial way, we obtain a parametrization of all $n$-dimensional simply connected complex homogeneous spaces with $q$-dimensional isotropy.  One may furthermore fix again an inner product $\ip$ on $\RR^n$ compatible with $J$ (i.e. $\la J\cdot,J\cdot\ra=\ip$) and require condition (h3) on $\ip$, in order to parameterize hermitian (or almost-hermitian if the integrability of $J$ is removed) homogeneous spaces  $(G_\mu/K_\mu,J,\ip)$.  Notice that the subset of those which are K\"ahler is just defined by extra polynomial conditions on $\mu$.  An analogous setting can be developed for symplectic, hypercomplex, and many other classes of geometric structures.  This approach has only been explored in the case of nilmanifolds (i.e. $q=0$ and $\mu$ nilpotent) in \cite{minimal}.  How do the convergence results obtained in Section \ref{convsec} fit into deformation theory of complex or symplectic manifolds?
\end{remark}

\section{Different notions of equivalence}\label{equivtypes}

The question of whether two given homogeneous spaces are isometric or not is usually a difficult task to handle, as it is the question on determining their diffeomorphism or even homeomorphism types.  There is a fourth natural equivalence relation between homogeneous spaces which involves their algebraic structure: $G/K$ and $G'/K'$ are said to be {\it equivariantly diffeomorphic} if there exists an isomorphism of Lie groups $\tilde{\vp}:G\longrightarrow G'$ such that $\tilde{\vp}(K)=K'$.  In that case, if $\vp:G/K\longrightarrow G'/K'$ is the corresponding {\it equivariant diffeomorphism} (i.e. $\vp\circ\pi=\pi'\circ\tilde{\vp}$), then
$$
\tau'(\tilde{\vp}(u))=\vp\tau(u)\vp^{-1}, \qquad\forall u\in G,
$$
that is, the actions of $G,G'$ on $G/K,G'/K'$, respectively, are equivalent or equivariant.  Two homogeneous spaces $(G/K,g_{\ip})$ and $(G'/K',g_{\ip'})$ are called {\it equivariantly isometric} if $g_{\ip}=\vp^*g_{\ip'}$ for some equivariant diffeomorphism $\vp:G/K\longrightarrow G'/K'$ (i.e. $\dif\vp|_{eK}$ is in addition an inner product space isometry between $(\pg,\ip)$ and $(\pg',\ip')$).

\begin{figure}
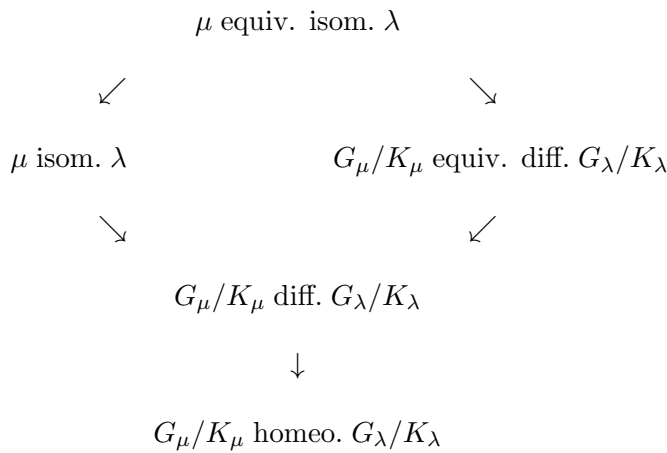

$$
\begin{array}{rcl}
& \mu \;\mbox{equiv. isom.}\; \lambda & \\ \\
\swarrow && \searrow \\ \\
\mu \;\mbox{isom.}\; \lambda &&
\hspace{-1.8cm} G_\mu/K_\mu \;\mbox{equiv. diff.}\; G_\lambda/K_\lambda \\ \\
\searrow && \swarrow \\ \\
& G_\mu/K_\mu \;\mbox{diff.}\; G_\lambda/K_\lambda & \\ \\
& \downarrow & \\ \\
& G_\mu/K_\mu \;\mbox{homeo.}\; G_\lambda/K_\lambda &
\end{array}
$$
\caption{Notions of equivalence by degree of generality}\label{equivs}
\end{figure}


In Figure \ref{equivs}, we have listed the equivalence relations between homogeneous spaces we have just mentioned according to their levels of generality.  All the converse assertions are false.  Aloff-Wallach spaces (see Example \ref{AW2}) provide examples of homeomorphic but nondiffeomorphic homogeneous spaces, as well as diffeomorphic homogeneous spaces which are not equivariantly diffeomorphic (see also Example \ref{ex1-52}).  On the other hand, certain nonabelian solvable Lie groups admit flat left-invariant metrics, providing examples of isometric homogeneous spaces which are not equivariantly isometric.

\begin{example}\label{AW2}
Let $W_{p,q}=\SU(3)/S^1_{p,q}$ be the Aloff-Wallach space described in Example \ref{AW}, and assume that $p,q\in\ZZ$ and are coprime.  It is well known that $W_{p,q}$ has fourth cohomology ring $H^4(W_{p,q},\ZZ)=\ZZ_r$, the cyclic group of order $r:=p^2+pq+q^2$ (see \cite[Lemma 3.3]{AlfWll}), showing that there are infinitely many homeomorphism classes among these spaces.  More precisely, if $s:=pq(p+q)$ then the following conditions must be added to $r=\tilde{r}$ in order to get the respective equivalence type between $W_{p,q}$ and $W_{\tilde{p},\tilde{q}}$:

\begin{itemize}
\item homotopy equivalent: $s\equiv\pm\tilde{s}\mod{r}$ (see \cite{Krg});

\item homeomorphic: $s\equiv\pm\tilde{s}\mod{2^3.3.r}$ (see \cite{KrcStl});

\item diffeomorphic: $s\equiv\pm\tilde{s}\mod{2^5.3.r}$ if $r$ is a multiple of $7$, and $\mod{2^5.3.7.r}$ otherwise (see \cite{Krg});

\item equivariantly diffeomorphic: $\{\tilde{p},\tilde{q},-(\tilde{p}+\tilde{q})\}=\{ p,q,-(p+q)\}$ (that is, the at most six possibilities of having $S^1_{p,q}$ and $S^1_{\tilde{p},\tilde{q}}$ conjugate in $\SU(3)$).
\end{itemize}

It was not a trivial task to find explicit pairs $(p,q)$, $(\tilde{p},\tilde{q})$ showing that none of the above equivalence types coincide for Aloff-Wallach spaces (see \cite{Krg,KrcStl}).
\end{example}

In what follows, we are interested in describing as simple as possible, for a given notion, the equivalence class of a homogeneous space $\mu\in\hca_{q,n}$ (see identification (\ref{hsmu})) as a subset of $\hca_{q,n}$.  There is a natural linear action of $\Gl_{q+n}(\RR)$ on $V_{q+n}$ given by
\begin{equation}\label{action}
h.\mu(x,y)=h\mu(h^{-1}x,h^{-1}y), \qquad x,y\in\RR^{q+n}, \quad h\in\Gl_{q+n}(\RR),\quad \mu\in V_{q+n}.
\end{equation}
The variety of Lie algebras $\lca_{q+n}$ is $\Gl_{q+n}(\RR)$-invariant, the Lie algebra isomorphism classes are
precisely the $\Gl_{q+n}(\RR)$-orbits and the isotropy subgroup $\Gl_{q+n}(\RR)_\mu$ equals $\Aut(\RR^{q+n},\mu)$ for any $\mu\in\lca_{q+n}$.

\begin{proposition}\label{const}
If $\mu\in\hca_{q,n}$ then $h.\mu\in\hca_{q,n}$ for any $h\in\Gl_{q+n}(\RR)$ of the form
\begin{equation}\label{formh}
h:=\left[\begin{smallmatrix} h_q&A\\ 0&h_n
\end{smallmatrix}\right]\in\Gl_{q+n}(\RR), \quad h_q\in\Gl_q(\RR), \quad h_n\in\Gl_n(\RR), \quad A:\RR^n\longrightarrow\RR^q,
\end{equation}
such that
\begin{equation}\label{adkh}
[h_n^th_n,\ad_{\mu}{\RR^q}|_{\RR^n}]=0,
\end{equation}
and
\begin{equation}\label{adkh2}
A\ad_{\mu}{z}|_{\RR^n}=h_q\ad_{\mu}{z}|_{\RR^q}h_q^{-1}A, \qquad\forall z\in\RR^q.
\end{equation}
In that case,  $G_{h.\mu}/K_{h.\mu}$ and $G_{\mu}/K_{\mu}$ are equivariantly diffeomorphic and $\left(G_{h.\mu}/K_{h.\mu},g_{h.\mu}\right)$ is equivariantly isometric to $\left(G_{\mu}/K_{\mu},g_{\la h_n\cdot,h_n\cdot\ra}\right)$.
\end{proposition}

\begin{remark}\label{const3}
It follows from the last assertion in the above proposition that the subspace
$$
\left\{ h.\mu:h_q=I,\, A=0,\,h_n\,\mbox{satisfies (\ref{adkh})}\right\}\subset\hca_{q,n},
$$
parameterizes the set of all $G_\mu$-invariant metrics on $G_\mu/K_\mu$.   Also notice that $\mu$ and $h.\mu$ have the same volume if $\det{h_n}=1$.
\end{remark}

\begin{proof}
For such an $h$, we must check that $\lambda:=h.\mu$ satisfies conditions (h1)-(h4) defining $\hca_{q,n}$ (see (\ref{hkn})).  We first note that
$$
h^{-1}:=\left[\begin{smallmatrix} h_q^{-1}&-h_q^{-1}Ah_n^{-1}\\ 0&h_n^{-1}
\end{smallmatrix}\right].
$$
Condition (h1) always holds since $h:(\RR^{q+n},\mu)\longrightarrow (\RR^{q+n},\lambda)$ is an isomorphism of Lie
algebras leaving $\RR^q$ invariant and it follows from (\ref{adkh2}) that for all $z\in\RR^q$, $x\in\RR^n$,
\begin{align}
\lambda(z,x)=& h_q\mu(h_q^{-1}z,-h_q^{-1}Ah_n^{-1}x) + A\mu(h_q^{-1}z,h_n^{-1}x) + h_n\mu(h_q^{-1}z,h_n^{-1}x) \notag \\ 
=& -h_q\ad_{\mu}(h_q^{-1}z)h_q^{-1}Ah_n^{-1}x + A\ad_{\mu}(h_q^{-1}z)h_n^{-1}x + h_n\mu(h_q^{-1}z,h_n^{-1}x)  \label{adkh3} \\
=& h_n\mu(h_q^{-1}z,h_n^{-1}x)\in\RR^n. \notag
\end{align}
We therefore obtain that $\ad_\lambda{z}|_{\RR^n}=h_n\ad_\mu{h_q^{-1}z}|_{\RR^n}h_n^{-1}$ for all $z\in\RR^q$, which
implies that (h4) holds for $\lambda$, and also that $\lambda$ satisfies (h3) if and only if (\ref{adkh}) holds.

There exists a unique isomorphism of Lie groups
$$
\tilde{\vp}:G_\mu\longrightarrow G_\lambda, \qquad\mbox{such that} \qquad
\dif\tilde{\vp}|_e=h.
$$
Since $\tilde{\vp}(K_\mu)$ is a connected Lie subgroup of $G_\lambda$  with Lie algebra $h(\RR^q)=\RR^q$, we have that $K_\lambda=\tilde{\vp}(K_\mu)$ and thus (h2) follows.

Concerning the last assertion, we have that the diffeomorphism
\begin{equation}\label{isom}
\vp:G_\mu/K_\mu\longrightarrow G_\lambda/K_\lambda, \qquad \vp(uK_\mu):=\tilde{\vp}(u)K_{\lambda}, \qquad\forall u\in G_\mu,
\end{equation}
is well defined and is an isometry between the homogeneous spaces $\left(G_{\mu}/K_{\mu},g_{\la h_n\cdot,h_n\cdot\ra}\right)$ and $\left(G_\lambda/K_\lambda,g_\lambda\right)$, as $\dif\vp|_{eK_\mu}$ coincides with the inner product space isometry
$$
h_n:(\RR^n,\la h_n\cdot,h_n\cdot\ra)\longrightarrow(\RR^n,\ip),
$$
under the natural identifications.
\end{proof}

\begin{corollary}\label{const2}
The group $\Gl_q(\RR)\times\Or(n)$ leaves the set $\hca_{q,n}$ invariant and $h.\mu$ and $\mu$ are equivariantly isometric homogeneous spaces for any $h\in\Gl_q(\RR)\times\Or(n)$, $\mu\in\hca_{q,n}$.
\end{corollary}

Let us now analyze condition (\ref{adkh}) more in detail.  The isotropy representation $\ad_\mu:\RR^q\longrightarrow\End(\RR^n)$ of a homogeneous space $\mu\in\hca_{q,n}$, which is faithful by (h4) and unitary by (h3), can be decomposed into isotypical components as
$$
\RR^n=V_1^{n_1}\oplus...\oplus V_r^{n_r},
$$
where $V_i,V_j$ are non-equivalent irreducible representations of the Lie algebra $(\RR^q,\mu)$ for all $i\ne j$, and $V_i^{n_i}\simeq V_i\oplus\dots\oplus V_i$ ($n_i$ times).  The space of intertwining operators is therefore given by
\begin{equation}\label{end}
\End_{\ad_\mu}(\RR^n)=\glg_{n_1}(\FF_1)\oplus...\oplus\glg_{n_r}(\FF_r),
\end{equation}
where $\FF_i=\RR,\CC$ or $\HH$ depending on the type of $V_i$.  Recall that the possible types of a real representation are {\it real}, {\it complex} or {\it quaternionic}, i.e. $\End_{\ad_\mu}(V_i)=\RR,\CC$ or $\HH$, respectively (see \cite{BrcDck}).  It follows from Proposition \ref{const} that for each $\mu\in\hca_{q,n}$, if $U_\mu$ is the subset of $\Gl_n(\RR)$ defined by
$$
U_\mu:=\{h_n\in\Gl_n(\RR):h_n^th_n\in\End_{\ad_\mu}(\RR^n)\},
$$
then $h.\mu\in\hca_{q,n}$ for any $h\in\tilde{U}_\mu$, where $\tilde{U}_\mu$ is the subset of $\Gl_{q+n}(\RR)$ given by
$$
\tilde{U}_\mu:=\left\{
\left[\begin{smallmatrix} h_q&A\\ 0&h_n
\end{smallmatrix}\right]\in\Gl_{q+n}(\RR) : h_q\in\Gl_q(\RR), \, h_n\in U_\mu, \, A \,\mbox{satisfies (\ref{adkh2})}\right\}.
$$
If we define
$$
\sym_{\FF}(m):= \{ A\in\glg_m(\FF):\overline{A^t}=A\},
$$
then by using (\ref{end}) and the polar decomposition one easily obtains that
$$
U_\mu=\Or(n)\left(U_1\times...\times U_r\right), \qquad U_i:=e^{\sym_{\FF_i}(n_i)}.
$$
Notice that $\tilde{U}_\mu.\hca_{q,n}$ is not necessarily contained in $\hca_{q,n}$, it only satisfies $\tilde{U}_\mu.\mu\subset\hca_{q,n}$, but we may consider for each faithful and unitary representation $\theta:(\RR^q,\mu)\longrightarrow\End(\RR^n)$ the subset of homogeneous spaces having $\theta$ as its isotropy representation, that is,
$$
\hca_{q,n}(\theta):=\{\mu\in\hca_{q,n}:\ad_\mu\RR^q|_{\RR^n}=\theta\}.
$$
Thus $U_\mu=U_\lambda$ for any $\mu,\lambda\in\hca_{q,n}(\theta)$, and so if we denote these subsets by $U_\theta$, then
$$
(\Gl_q(\RR)\times U_\theta).\hca_{q,n}(\theta)\subset\hca_{q,n}(\theta).
$$

\begin{proposition}\label{eqdif}
$G_\mu/K_\mu$ and $G_\lambda/K_\lambda$, $\mu,\lambda\in\hca_{q,n}$, are equivariantly diffeomorphic if and only if $\lambda\in\tilde{U}_\mu.\mu$.
\end{proposition}

\begin{proof}
If $\vp:G_\mu/K_\mu\longrightarrow G_\lambda/K_\lambda$ is an equivariant diffeomorphism determined by an isomorphism $\tilde{\vp}:G_\mu\longrightarrow G_\lambda$ and $h:=\dif\tilde{\vp}|_{e}$, then $\lambda=h.\mu$ and $h\RR^q\subset\RR^q$ follows from the fact that $\tilde{\vp}(K_\mu)=K_\lambda$.  We now use that $\lambda(\RR^q,\RR^n)\subset\RR^n$ to obtain from (\ref{adkh3}) that $A:\RR^n\longrightarrow\RR^q$, the $\RR^q$-component of $h|_{\RR^n}$, must satisfy condition (\ref{adkh2}).  Finally, it follows from the fact that $\ad_\lambda{z}|_{\RR^n}=h_n\ad_\mu{h_q^{-1}z}|_{\RR^n}h_n^{-1}$ for all $z\in\RR^q$ and condition (h3) that $h_n$ satisfies (\ref{adkh}), which implies that $h\in\tilde{U}_\mu$.

The converse assertion is the content of Proposition \ref{const}.
\end{proof}

Summarizing, we have that
\begin{itemize}
\item The group $\Gl_q(\RR)\times\Or(n)$ acts on $\hca_{q,n}$ in such a way that all the elements in the same orbit are pairwise equivariantly isometric.

\item  For any $\mu\in\hca_{q,n}$, the subset $U_\mu.\mu\subset\hca_{q,n}$ parameterizes the set of all $G_\mu$-invariant metrics on $G_\mu/K_\mu$, where we  embed $U_\mu\hookrightarrow\Gl_{q+n}(\RR)$ in the usual way (see Remark \ref{const3}).

\item For any $\mu\in\hca_{q,n}$, the subset $\tilde{U}_\mu.\mu\subset\hca_{q,n}$ consists of those elements in $\hca_{q,n}$ which are equivariantly diffeomorphic to $G_\mu/K_\mu$.

\item The subsets $(\Gl_q(\RR)\times U_\theta).\mu$, $\mu\in\hca_{q,n}(\theta)$, are precisely the equivariant diffeomorphism classes inside $\hca_{q,n}(\theta)$.
\end{itemize}

\begin{example}\label{ex1-52}
For $\mu=\mu_{p,q,a,b,c,d,e,f}\in\hca_{1,5}$ given in Example \ref{ex1-5} we have that
$$
\ad_\mu{e_0}=\left[\begin{smallmatrix} 0&&&&\\ &0&-p&&\\ &p&0&&\\ &&&0&-q\\ &&&q&0
\end{smallmatrix}\right],
$$
and hence
$$
U_\mu=\left\{
\begin{array}{lcl}
\Or(5)\left[\begin{smallmatrix} r&&\\ &\Gl_1(\CC)&\\ &&\Gl_1(\CC)\end{smallmatrix}\right] =\Or(5)\left[\begin{smallmatrix} r_1&&\\ &r_2I_{2\times 2}&\\ &&r_3I_{2\times 2}\end{smallmatrix}\right], && 0<p<q; \\ \\
\Or(5)\left[\begin{smallmatrix} \Gl_3(\RR)&\\ &\Gl_1(\CC)&\end{smallmatrix}\right] =\Or(5)\left[\begin{smallmatrix} e^{\sym(3)}\\ &rI_{2\times 2}&\end{smallmatrix}\right], && 0=p<q; \\ \\
\Or(5)\left[\begin{smallmatrix} r&\\ &\Gl_2(\CC)\end{smallmatrix}\right] =\left[\begin{smallmatrix} r&\\ &e^{\im\ug(2)}\end{smallmatrix}\right], && 0<p=q.
\end{array}\right.
$$
If $G_\mu\simeq\SU(2)\times\SU(2)$, then the homogeneous spaces $G_\mu/K_\mu$ are all diffeomorphic to $S^3\times S^2$, but if $0\leq p\leq q$ then two different values of $p/q\in[0,1]\cap\QQ$ give rise to non-equivalent $G_\mu$-actions on $S^3\times S^2$ (see \cite[Example 6.8]{BhmWngZll}), that is, to non-equivariantly diffeomorphic homogeneous spaces.
\end{example}

\section{Curvature invariants}\label{curvinv}

In this section, we describe a quite intriguing necessary and sufficient condition for two homogeneous spaces $\mu,\lambda\in\hca_{q,n}$ be isometric.  The condition is in the spirit of invariant theory and was proved by I.M. Singer in \cite{Sng} (see \cite{NclTrc} and \cite{PrfTrcVnh} for further information).  These results are being used in some work in progress on homogeneous Ricci solitons.

Let $\nabla_\mu$ denote the Levi-Civita connection and $\Riem_\mu$ the corresponding Riemannian curvature tensor of $\mu\in\hca_{q,n}$.  Recall that any $\mu\in\hca_{q,n}$ is identified with the homogeneous space $(G_\mu/K_\mu, g_\mu)$ according to (\ref{hsmu}).  By identifying $\RR^n$ with the corresponding Killing vector fields of $G_\mu/K_\mu$, it follows that $\Riem_\mu$ is determined by its value at $eK_\mu$, the $4$-linear map given by
$$
\Riem_\mu:=\Riem(g_{\mu})(eK_\mu):\RR^n\times\RR^n\times\RR^n\times\RR^n \longrightarrow\RR.
$$
In the same way, the covariant derivative $\nabla_\mu^k\Riem_\mu$ can be viewed as a vector in $\otimes^{4+k}(\RR^n)^*$ for any $k\geq 0$ ($\nabla_\mu^0\Riem_\mu:=\Riem_\mu$), and we consider for each $\mu\in\hca_{q,n}$ the vector
$$
w_\mu:= (\Riem_\mu, \nabla_\mu\Riem_\mu, ...,\nabla_\mu^m\Riem_\mu)\in W:=\bigoplus_{k=0}^m\left(\otimes^{4+k}(\RR^n)^*\right), \quad m:=\tfrac{n(n-1)}{2}-1.
$$
If $\mu,\lambda\in\hca_{q,n}$ are isometric, then the isometry $\vp:G_\mu/K_\mu\longrightarrow G_\lambda/K_\lambda$ can be assumed to satisfy $\vp(eK_\mu)=eK_\lambda$.  Thus $h:=\dif\vp|_{eK_\mu}\in\Or(n)$ and we have that $h.\nabla_\mu^k\Riem_\mu=\nabla_\lambda^k\Riem_\lambda$ for all $k$, where the actions of $\Or(n)\subset\Gl_n(\RR)$ on the different tensorial vector spaces are the standard ones. This implies that
\begin{equation}\label{invf}
w_\lambda\in\Or(n).w_\mu, \qquad\forall\mu,\lambda\in\hca_{q,n} \quad \mbox{isometric}.
\end{equation}
Let us now take $f\in\RR[W]^{\Or(n)}$, that is, a polynomial function $f:W\longrightarrow\RR$ which is $\Or(n)$-invariant (i.e. $f(h.w)=f(w)$ for all $h\in\Or(n)$, $w\in W$).  We also denote by $f$ the function
\begin{equation}\label{invariant}
f:\hca_{q,n}\longrightarrow\RR, \qquad f(\mu):=f(w_\mu),
\end{equation}
which is also polynomial on $\mu$.  We call such an $f$ a {\it curvature invariant}, as it follows from (\ref{invf}) that $f(\mu)=f(\lambda)$ for any pair $\mu,\lambda\in\hca_{q,n}$ of isometric homogeneous spaces.  The converse assertion is a very nice and important result in homogeneous geometry proved in \cite{Sng} (see also \cite[Theorem 2.5]{NclTrc} for an alternative proof and \cite[Theorem 2.3]{PrfTrcVnh}).

\begin{theorem}\label{singer}
The following assertions are equivalent:
\begin{itemize}
\item[(i)] $\mu,\lambda\in\hca_{q,n}$ are isometric.

\item[(ii)] $f(\mu)=f(\lambda)$ for any $f\in\RR[W]^{\Or(n)}$.

\item[(iii)] $w_\lambda\in\Or(n).w_\mu$.
\end{itemize}
\end{theorem}

The equivalence between (ii) and (iii) actually follows from a strong result in invariant theory: $\RR[W]^{\Or(n)}$ separates orbits as $\Or(n)$ is compact.  Since $\Or(n)$ is a reductive group, another classical theorem from invariant theory states that $\RR[W]^{\Or(n)}$ is finitely generated as an algebra, say
$$
\RR[W]^{\Or(n)}=\left\la f_1,...,f_r\right\ra.
$$
By considering $F:=(f_1,...,f_r):\hca_{q,n}\longrightarrow\RR^r$, we conclude from Theorem \ref{singer} that

\begin{quote}
$\mu,\lambda\in\hca_{q,n}$ are isometric if and only if $F(\mu)=F(\lambda)$.
\end{quote}

In other words, the isometry classes in $\hca_{q,n}$ are precisely the level sets of a polynomial function $F:\hca_{q,n}\longrightarrow\RR^r$.

\begin{example}
A family of curvature invariants whose computation is usually doable is $f_k(\mu):=\tr{\Ricci_\mu^k}$, where $\Ricci_\mu$ is the Ricci operator of $\mu\in\hca_{q,n}$.  Recall that the values of $f_1,...f_n$ at $\mu$ actually determines the set of Ricci eigenvalues (counting multiplicities).  As a homogeneous manifold is flat if and only if it is Ricci flat (see \cite{AlkKml}), the flat elements in $\hca_{q,n}$ can be characterized by a single polynomial equation: $f_2(\mu)=0$.
\end{example}

The setting described in this section motivates the definition of a distance on $\hca_{q,n}$ given by
$$
d(\mu,\lambda):=d_W(\Or(n).w_\mu,\Or(n).w_\lambda) =\min\{d_W(h.w_\mu,h'.w_\lambda):h,h'\in\Or(n)\},
$$
where $d_W$ is the euclidean distance in $W$.  We may also consider the Hausdorff distance between compact subsets of $W$, but this will be equivalent since the subsets involved are orbits by a group of isometries of $W$.  It follows from Theorem \ref{singer} that $d(\mu,\lambda)=0$ if and only if $\mu$ and $\lambda$ are isometric as homogeneous manifolds.  If $\mu_k\to\lambda$ in $V_{q+n}$, as $k\to\infty$, then $d(\mu_k,\lambda)\to 0$, and hence the topology of the metric space $(\hca_{q,n},d)$ is weaker than the induced on $\hca_{q,n}$ by the usual vector space topology of $V_{q+n}$.  We note that these topologies are not equivalent, it may for instance happen that $\Or(n).w_{\mu_k}\to\{ 0\}=\Or(n).w_0$, and nevertheless $\mu_k\to\lambda\ne 0$ (e.g. take the sequence $\mu_k:=\mu_{1+1/k,1-1/k,0}$ in Example \ref{ex0-3} of nonflat metrics on $E(2)$ converging to the flat manifold $\lambda:=\mu_{1,1,0}$).

\section{Convergence}\label{convsec}

In this section, all manifolds are assumed to be connected and all Riemannian metrics to be smooth (i.e. $C^\infty$) and complete.

\subsection{General case} Let $M$ be a differentiable manifold.  A sequence $g_k$ of Riemannian metrics on $M$ is said to {\it converge} (smoothly) to a Riemannian metric $g$ as $k\to\infty$ (denoted by $g_k\to g$) if for all compact subsets $K\subset M$, the tensor $g_k-g$ and its covariant derivatives of all orders (with respect to any fixed background connection) each converge uniformly to zero on $K$.

\begin{remark}\label{convr}
By using charts with relatively compact domains which cover $M$, convergence $g_k\to g$ can be rephrased as follows: the partial derivative $\partial^\alpha(g_k)_{ij}$ of the coordinates $(g_k)_{ij}$ of the metrics converge to $\partial^\alpha g_{ij}$ uniformly, as $k\to\infty$, for every chart and every multiindex $\alpha$.
\end{remark}

A {\it pointed Riemannian manifold} $(M,g,p)$ is simply a Riemannian manifold $(M,g)$ with a point $p\in M$, which plays the role of basepoint or point of reference.  Two $(M,g,p)$, $(M',g',p')$ are called {\it isometric} if there is an isometry $\vp:(M,g)\longrightarrow (M',g')$ such that $\vp(p)=p'$.

\begin{definition}\label{pointed}
(Smooth pointed or Cheeger-Gromov topology)  A sequence $(M_k,g_k,p_k)$ of pointed Riemannian manifolds is said to {\it converge in the pointed sense} to a pointed Riemannian manifold $(M,g,p)$ as $k\to\infty$ if there exist
\begin{itemize}
\item[ ] a sequence of open subsets $\Omega_k\subset M$ containing $p$, so that any compact subset of $M$ eventually lies in all $\Omega_k$ for sufficiently large $k$;

\item[ ] a sequence of smooth maps $\phi_k:\Omega_k\longrightarrow M_k$ which are diffeomorphisms onto open subsets $\Lambda_k\subset M_k$ (i.e. {\it embeddings}) and satisfy $\phi_k(p)=p_k$ for all $k$;
\end{itemize}
such that $\phi_k^*g_k\to g$ smoothly as $k\to\infty$ on $M$ (or more precisely, on every compact subset of $M$).
\end{definition}

Some remarks on this topology may be in order (see e.g. \cite[Chapter 8$^+$]{Grm}, \cite[Chapter 10]{Ptr}, \cite[Chapter 4]{libro}, \cite[Section 7.1]{Tpp} and \cite[Chapter 9]{AndHpp} for further information).  Assume that $(M_k,g_k,p_k)$ converges in the pointed sense to $(M,g,p)$ as $k\to\infty$.

\begin{itemize}
\item If $M$ is compact then $\phi_k:M\longrightarrow M_k$ is a diffeomorphism for all $k$ (as $\phi_k(M)$ is open and closed in $M$).  Thus the basepoints play no role in the pointed convergence, which in this case just means that $(M_k,g_k)$ converges smoothly to $(M,g)$ up to pullback by diffeomorphisms.

\item On the contrary, the example of the Rosenau metrics (i.e. longer and longer cigars converging to a cylinder, cf. \cite[9.2.2]{AndHpp}) shows that $M$ can be noncompact and non-simply connected, even when all the manifolds $M_k$ are compact and simply connected.

\item Also, the location of the basepoints can be crucially involved in the convergence when $M$ is noncompact: if $g_1$ is a metric on $\RR^n$ which coincides with the flat metric $g_0$ outside a compact set, then $(\RR^n,g_1,p_k)\to(\RR^n,g_0,0)$ if $p_k\to\infty$, but $(\RR^n,g_1,p_k)\to(\RR^n,g_1,p)$ if $p_k=p$ for all $k$ (see also the first example in \cite[Section 7.1]{Tpp} and \cite[Figure 9.3]{AndHpp}).

\item It is easy to check that the distances satisfy
$$
d_{g_k}(\phi_k(q),\phi_k(q')) \to d_g(q,q'), \qquad\forall q,q'\in M,
$$
from which it follows that for any $r>0$ the metric balls satisfy $B_g(p,r)\subset\Omega_k$ and $B_{g_k}(p_k,r)\subset\Lambda_k$ for sufficiently large $k$ (recall that metric balls are compact due to completeness).

\item The limit $(M,g,p)$ is unique up to isometry.

\item The following two conditions must hold;
\begin{itemize}
\item {\it bounded geometry}: for all $r>0$ and $j\in\ZZ_{\geq 0}$,
\begin{equation}\label{cotacurv}
\sup_k\;\sup_{B_{g_k}(p_k,r)} \|\nabla_{g_k}^j\Riem(g_k)\|_{g_k}<\infty,
\end{equation}
where $\nabla_{g_k}$ is the Levi-Civita connection and $\|\cdot\|_{g_k}$ denotes the corresponding norm in the spaces of sections of the different tensor bundles over $M_k$;
\item[ ]
\item {\it non-collapsing}:
\begin{equation}\label{cotainj}
\inf_k \inj(M_k,g_k,p_k)>0,
\end{equation}
where $\inj(M_k,g_k,p_k)$ is the injectivity radius of $(M_k,g_k)$ at $p_k$.
\end{itemize}
\end{itemize}

Recall that the {\it injectivity radius} of $(M,g)$ at $p$ is the largest $\epsilon$ for which the exponential map $\exp_p:B(0,\epsilon)\longrightarrow B_g(p,\epsilon)$ is a diffeomorphism, where $B(0,\epsilon)=\{ x\in\tang_pM:g_p(x,x)<\epsilon^2\}$.  The following result is considered the fundamental theorem of convergence theory of Riemannian manifolds.

\begin{theorem}\label{pcomp} (Compactness)
Let $(M_k,g_k,p_k)$ be a sequence of complete pointed Riemannian manifolds of dimension $n$ satisfying {\rm (\ref{cotacurv})} and {\rm (\ref{cotainj})}.  Then there exists a subsequence of $(M_k,g_k,p_k)$ which converges to a complete pointed Riemannian manifold $(M,g,p)$ of dimension $n$ in the pointed topology.
\end{theorem}

A proof of this theorem can be found in \cite[10.3-10.4]{Ptr}, \cite[Chapter 4]{libro} and \cite[2.3]{Hml}, and its origins can be traced back to ideas of Gromov \cite{Grm} and Cheeger.  We note that the finiteness of the number of diffeomorphism classes follows on any subset of compact Riemannian manifolds where a compactness (or precompactness) theorem can be applied (recall that a sequence of pairwise non-diffeomorphic manifolds can never subconverge to a compact limit).

\subsection{Homogeneous case}  If $g_k$ is a sequence of homogeneous metrics on a differentiable manifold $M$ such that $g_k$ smoothly converges to a metric $g$, then $g$ is also homogeneous.  Indeed, given $p,q\in M$ there exists for each $k$ a $g_k$-isometry $\vp_k$ such that $\vp_k(p)=q$, and it follows from $g_k\to g$ that the set $\{\vp_k\}$ is locally uniformly bounded and equicontinuous.  Hence, by the Arzela-Ascoli Theorem, a subsequence converges locally uniformly to a continuous map $\vp:M\longrightarrow M$ which is automatically an isometry of $(M,g)$ as $\vp$ preserves its Riemannian distance.  Although the set of all isometry classes of metrics on a given noncompact $M$ endowed with the quotient smooth topology is not Hausdorff, it is proved in \cite[Sections 6.1,6.2]{Hbr} that, on the contrary, the subset of those classes which are homogeneous is so, by applying an Arzela-Ascoli argument as above.

We are interested here in pointed convergence of homogeneous manifolds.  Special features for this case are hard to find in the literature.  For a strong use of the pointed topology of compact homogeneous manifolds we refer to the proofs of Theorems 1.1 and 2.1 in \cite{BhmWngZll}.

A few comments are in order at this point.  Let us assume that all $(M,g_k)$ are homogeneous and that $(M_k,g_k,p_k)$ converges in the pointed sense to $(M,g,p)$ as $k\to\infty$.

\begin{itemize}
\item Two pointed homogeneous manifolds are isometric if and only if they are isometric in the usual sense.

\item The limit $(M,g)$ is homogeneous.  Indeed, given $q\in M$, we can assume that $p,q\in\Omega_k$ for all $k$ and define $f_k:=\phi_k^{-1}\circ h_k\circ\phi_k$, where $h_k\in\Iso(M_k,g_k)$ satisfies $h_k(p_k)=\phi_k(q)$.  Thus $f_k(p)=q$ for all $k$ and by an Arzela-Ascoli argument together with a diagonal procedure one gets a limit $f:M\longrightarrow M$ with $f(p)=q$, which automatically satisfies $f\in\Iso(M,g)$ by using that $d_{g_k}(\phi_k(a),\phi_k(b))\to d_g(a,b)$ for all $a,b\in\Omega_k$.

\item The location of the basepoints $p_k$ and $p$ play no role in the pointed convergence, in the sense that we can change all of them by any other sequence $q_k\in M_k$ and $q\in M$ and use homogeneity.  However, unlike the compact case, $M$ being nonhomeomorphic to $M_k$ for all $k$ is a possible behavior (e.g. a sequence of expanding spheres converges to the plane in the pointed topology).

\item It may also happen in the homogeneous case that all $M_k$ are simply connected but $M$ is not.  Take for instance the sequence $g_k$ of left-invariant metrics on $S^3$ obtained by scaling times $k$ the round metric on the orthogonal complement of any fixed direction $X\in\sug(2)$.  It is not very hard to check that conditions (\ref{cotacurv}) and (\ref{cotainj}) hold for $(S^3,g_k)$, and thus there must be a subsequence converging to a homogeneous manifold $(M,g)$ in the pointed sense by the compactness theorem, which is easily seen to be flat.  Since $\gamma(t)=e^{tX}$ is a closed geodesic of $(S^3,g_k)$ having the same length for all $k$, it follows that $(M,g)$ must have a closed geodesic as well and so $M$ can not be simply connected (see Example \ref{berger} for a more detailed treatment of this example, where it is proved that the pointed limit is indeed $S^1\times\RR^2$).  The manifolds $(S^3,g_k)$ are called {\it Berger spheres} in the literature, and the sequence $(S^3,\tfrac{1}{k}g_k)$ is a famous example of collapsing (toward $S^2$) with bounded curvature.
\end{itemize}

There are two other notions of convergence (infinitesimal and local) which naturally arise in studying the space of homogeneous manifolds and where the topology of the manifolds is much less involved.

\begin{definition}\label{infconv} (infinitesimal) A sequence $(M_k,g_k)$ of homogeneous manifolds is said to {\it infinitesimally converge} to a homogeneous manifold $(M,g)$ as $k\to\infty$ if there exist
\begin{itemize}
\item[ ] a sequence of open subsets $\Omega_k\subset M$ containing a point $p\in M$;

\item[ ] a sequence of embeddings $\phi_k:\Omega_k\longrightarrow M_k$;
\end{itemize}
such that $\phi_k^*g_k\to g$ smoothly as $k\to\infty$ {\it at} $p$, in the sense that for any $\epsilon>0$, there exists $k_0=k_0(\epsilon)$ such that for $k\geq k_0$,
$$
\sup_{\Omega_k} \|\nabla_g^j(\phi_k^*g_k-g)\|_g<\epsilon, \qquad\forall j\in\ZZ_{\geq 0}.
$$
\end{definition}

As in the homogeneous case one only needs to control covariant derivatives up to a finite order (see Section \ref{curvinv}), it is enough for the infinitesimal convergence the existence of a $k_0(\epsilon,j)$ satisfying the required property for each fixed order $j$.  We also note that the point $p$ can be changed by any other point in $M$ due to homogeneity. The infinitesimal convergence of homogeneous manifolds is somewhat weak, notice that it does not require any condition on the size of the $\Omega_k$'s and so actually only the germs of the metrics at $p$ are involved.  The injectivity radius may therefore go to zero and it is even possible that all manifolds $M_k$, $M$ be pairwise non-homeomorphic, as Example \ref{AW3} shows.

\begin{definition}\label{locconv} (local) A sequence $(M_k,g_k)$ of homogeneous manifolds is said to {\it locally converge} to a homogeneous manifold $(M,g)$
as $k\to\infty$ if there exist
\begin{itemize}
\item[ ] a nonempty open subset $\Omega\subset M$;

\item[ ] a sequence of embeddings $\phi_k:\Omega\longrightarrow M_k$;
\end{itemize}
such that $\phi_k^*g_k\to g$ smoothly as $k\to\infty$ on $\Omega$.
\end{definition}

Notice that the open subset $\Omega$ can be assumed to contain any point $p\in M$ by using homogeneity.  It follows at once from the definitions that the three notions of convergence of homogeneous manifolds are related by

\begin{center}
pointed $\quad\Rightarrow\quad$ local $\quad\Rightarrow\quad$ infinitesimal.
\end{center}

Actually, the only difference between these three definitions of convergence lies in the size of the open subsets in the sequence $\Omega_k\subset M$:

\begin{itemize}
\item[ ] {\it infinitesimal}: no condition, $\Omega_k$ can be arbitrarily small (e.g. $\bigcap\Omega_k=\{ p\}$).

\item[ ] {\it local}: $\Omega_k$ stabilizes, i.e. $\Omega_k\supset\Omega\ne\emptyset$ for sufficiently large $k$.

\item[ ] {\it pointed}: $\Omega_k$ exhausts $M$, i.e. it eventually contains any given compact subset of $M$.
\end{itemize}

Both converse assertions are false: non-local infinitesimal convergence and non-pointed local convergence can be shown to occur (see Examples \ref{AW3} and \ref{berger}).

\subsection{Algebraic convergence}  Our aim in what follows is to study until what extent is the algebraic side of a homogeneous manifold involved in convergence issues.  In Section \ref{varhs}, we have defined a subset $\hca_{q,n}$ of the variety of Lie algebras which parameterizes the set of all $n$-dimensional simply connected homogeneous spaces with $q$-dimensional isotropy.  The space $\hca_{q,n}$ inherits the usual vector space topology from $V_{q+n}$, and a first natural question therefore arises: what kind of convergence of Riemannian manifolds this topology corresponds to?

Before starting with a rather technical matter, let us point out some useful facts:

\begin{itemize}
\item As for the other notions of convergence, a quick inspection of the examples in Section \ref{varhs} shows that both the topology and the Lie structure may also drastically change in the limit for the usual convergence of brackets (e.g. in Example \ref{ex0-3}, $\mu_{1,1/k,1/k}$ is a sequence of metrics on the simple Lie group $\SU(2)=S^3$ that converges to $\mu_{1,0,0}$, a metric on the Heisenberg Lie group, which is nilpotent and diffeomorphic to $\RR^3$).

\item For any $\mu\in\hca_{q,n}$ we can define a sequence $\mu_k\in\hca_{q,n}$ by $\mu_k|_{\RR^q\times\RR^{q+n}}:=\mu$, $\mu_k|_{\RR^n\times\RR^{n}}:=\tfrac{1}{k}\mu$, which converges to a flat element $\lambda\in\hca_{q,n}$ (recall that $\lambda$ is of the form $(K\ltimes\RR^n)/K$ for some compact subgroup $K\subset\Or(n))$.

\item Since in the homogeneous case it is enough to control the curvature tensors and their covariant derivatives at a single point, and since they all depend continuously on $\mu$ (see Section \ref{curvinv}), it follows that the usual convergence $\mu_k\to\lambda$ implies that the sequence $\left(G_{\mu_k}/K_{\mu_k},g_{\mu_k}\right)$ has bounded geometry (see (\ref{cotacurv})).

\item On the other hand, if $\mu_k$ locally converges to $\lambda$ (see Definition \ref{locconv}), then the sequence satisfies the non-collapsing condition  (see (\ref{cotainj})).  But under local convergence, bounded geometry also follows easily.  We therefore conclude from the compactness theorem that any locally convergent sequence $\mu_k$ must have a subsequence converging to a homogeneous manifold in the pointed topology.

\item Pointed or local subconvergence may however not follow from just the usual convergence of Lie brackets $\mu_k\to\lambda$, as Example \ref{AW3} shows.
\end{itemize}

\begin{example}\label{AW3}
Let $\mu_{p.q}$ denote the Lie bracket $\mu_{p,q,1,1,1,1}$ from Example \ref{AW}.  We consider the sequence of Aloff-Wallach spaces $\mu_k:=\mu_{1,1+1/k}$, which converges to $\mu_{1,1}$ in $\hca_{1,7}$, as $k\to\infty$.  However, the sequence $(W_{1,1+1/k},g_{\mu_k})=\left(G_{\mu_k}/K_{\mu_k},g_{\mu_k}\right)$ is certainly not converging in the pointed topology to $(W_{1,1},g_{\mu_{1,1}})$ since the manifolds $W_{1,1+1/k}=W_{k,k+1}$ are pairwise nonhomeomorphic (see Example \ref{AW2}) and $W_{1,1}$ is compact.  Since pointed convergence is not possible for any subsequence, we conclude again from the compactness theorem that $\inf\limits_k\inj(W_{1,1+1/k},g_{\mu_k})=0$ (recall that condition (\ref{cotacurv}) holds by the fact that $\mu_k\to\mu_{1,1}$), and so $\mu_k$ does not locally converge to $\mu_{1,1}$ either.
\end{example}

Let $\left(G_\mu/K_\mu,g_\mu\right)$ be the homogeneous space associated to $\mu\in\hca_{q,n}$, as in (\ref{hsmu}).

\begin{definition}\label{lieinj}
The {\it Lie injectivity radius} of $\left(G_\mu/K_\mu,g_\mu\right)$ is the largest $r_\mu>0$ such that
$$
\psi_\mu:=\pi_\mu\circ\exp_\mu: B(0,r_\mu)\longrightarrow G_\mu/K_\mu,
$$
is a diffeomorphism onto its image, where $\exp_\mu:\RR^{q+n}\longrightarrow G_{\mu}$ is the Lie exponential map, $\pi_\mu:G_\mu\longrightarrow G_\mu/K_\mu$ is the usual quotient map and $B(0,r_\mu)$ denotes the euclidean ball of radius $r_\mu$ in $\RR^n$.
\end{definition}

In other words, $B(0,r_\mu)$ is the largest ball where the canonical coordinates $\psi_\mu$ are defined.
\begin{remark}
The Lie injectivity radius can of course be defined for a homogeneous space $(G/K,g_{\ip})$ in its classical presentation, say with $\Ad(K)$-invariant decomposition $\ggo=\kg\oplus\pg$ (see Section \ref{hm}): just use balls in $(\pg,\ip)$.  Notice that the Lie injectivity radius depends on both the Lie theoretical data of $G/K$ and the Riemannian metric $g_{\ip}$.
\end{remark}

Every $\mu\in\hca_{q,n}$ uniquely determines a metric on a neighborhood of $0\in\RR^n$ as follows.  By setting $U_\mu:=\psi_\mu(B(0,r_\mu))$, we can associate to $\mu$ a metric $\tilde{g}_\mu$ on $B(0,r_\mu)$ given by
\begin{equation}\label{defgmu}
\tilde{g}_\mu:=\psi_\mu^*\left(g_\mu|_{U_\mu}\right).
\end{equation}

The metric $\tilde{g}_\mu$ on $B(0,r_\mu)$ does not depend on $G_\mu$, we can actually take any Lie group $G_\mu$ with Lie algebra $(\RR^{q+n},\mu)$, not necessarily simply connected, as long as the connected Lie subgroup $K_\mu$ with Lie algebra $\RR^q$ be closed in $G_\mu$.  Moreover, what $\tilde{g}_\mu$ really represents is a locally homogeneous structure, which happens to depend only on $\mu\in\hca_{q,n}$.  This will become quite clear in Proposition \ref{gmucc} below.

It will be useful to have an expression for the metric $\tilde{g}_\mu$ in terms of the canonical global chart $(x_1,...,x_n)$ of $B(0,r_\mu)\subset\RR^n$.  For a multiindex $\alpha=(\alpha_1,...,\alpha_n)$, we denote by $x^\alpha$ the monomial $x_1^{\alpha_1}...x_n^{\alpha_n}$, where $x=(x_1,...,x_n)\in\RR^n$ and $|\alpha|:=\alpha_1+...+\alpha_n$.  By `polynomial on $\mu$' we will always mean polynomial on the coordinates $\mu_{ij}^k$'s of $\mu$ defined by
$$
\mu(e_i,e_j)=\sum_{k=1}^{q+n}\mu_{ij}^ke_k,
$$
where $\{ e_i\}$ is the canonical basis of $\RR^{q+n}$.

\begin{proposition}\label{gmucc}
For each $\mu\in\hca_{q,n}$, the coordinate $(\tilde{g}_\mu)_{ij}$ of the metric $\tilde{g}_\mu$ is a real analytic function on $x$,
$$
(\tilde{g}_\mu)_{ij}(x)=\sum_\alpha a^{ij}_\alpha(\mu)\, x^\alpha, \qquad 1\leq i,j\leq n,
$$
which converges absolutely for $x\in B(0,r_\mu)\subset\RR^n$, where $r_\mu$ is the Lie injectivity radius of $\mu$.  Each coefficient $a^{ij}_\alpha$ is a universal polynomial expression on $\mu$ homogeneous of degree $|\alpha|$, depending only on $i,j$, $\alpha$, $q$ and $n$.  The lower terms are given by
\begin{align*}
(\tilde{g}_\mu)_{ij}(x)= & \delta_{ij} -\unm\sum_{k=1}^n(\mu_{q+k,q+j}^{q+i} +\mu_{q+k,q+i}^{q+j})x_k \\ & +\sum_{k,l=1}^n\Big(\unc\sum_{s=1}^n\mu_{q+k,q+i}^{q+s}\mu_{q+l,q+j}^{q+s} +\tfrac{1}{6}\sum_{r=1}^{q+n}\mu_{q+k,r}^{q+i}\mu_{q+l,q+j}^r+\mu_{q+k,r}^{q+j}\mu_{q+l,q+i}^r\Big)\, x_kx_l  \\
& + \;\mbox{monomials of degree}\, \geq 3.
\end{align*}
\end{proposition}

\begin{proof}
We start by recalling the formula for the derivative of the exponential map $\exp_\mu:\RR^{q+n}\longrightarrow G_\mu$ (see for instance \cite[2.14.3]{Vrd}), given by
\begin{equation}\label{dexp}
\dif\exp_\mu|_x=\dif L_{\exp_\mu(x)}|_e\circ \frac{I-e^{-\ad_\mu{x}}}{\ad_\mu{x}}, \qquad\forall x\in\RR^{q+n},
\end{equation}
where $\RR^{q+n}$ is identified with the tangent space at $x$, $L_u$ denotes left multiplication by $u$ on $G_\mu$ and
$$
\frac{I-e^{-\ad_\mu{x}}}{\ad_\mu{x}} := \sum_{k=0}^\infty \tfrac{(-1)^k}{(k+1)!}(\ad_\mu{x})^k = I-\unm\ad_\mu{x}+\tfrac{1}{6}(\ad_\mu{x})^2-\tfrac{1}{24}(\ad_\mu{x})^3+ ...
$$
Since the $ij$-entry of the matrix of $\ad_\mu{x}$ with respect to the basis $\{ e_1,\dots,e_{q+n}\}$ is given by $(\ad_\mu{x})_{ij}=\sum\limits_k\mu_{kj}^ix_k$, we have that
\begin{equation}\label{dexp2}
\left(\frac{I-e^{-\ad_\mu{x}}}{\ad_\mu{x}}\right)_{ij}=\sum_\alpha b_\alpha^{ij}(\mu)x^\alpha, \qquad 1\leq i,j\leq q+n,
\end{equation}
where $b_\alpha^{ij}$ is polynomial on $\mu$ of degree $|\alpha|$.  If we set $u:=\exp_\mu(x)$ and $A:=\frac{I-e^{-\ad_\mu{x}}}{\ad_\mu{x}}$ for short, then it follows from (\ref{dexp}), equality $\pi_\mu\circ L_u=\tau(u)\circ\pi_\mu$ and (\ref{gmu}) that

\begin{align*}
(\tilde{g}_\mu)_{ij}(x) &= g_\mu(\psi_\mu(x))\left(\dif\psi_\mu|_xe_{q+i}, \dif\psi_\mu|_xe_{q+j}\right) \\
&= g_\mu(uK_\mu)\left(\dif\pi_\mu|_{u}\dif \exp_\mu|_xe_{q+i},\dif\pi_\mu|_{u}\dif \exp_\mu|_xe_{q+j}\right) \\
&= g_\mu(uK_\mu)\left(\dif\pi_\mu|_{u}\dif L_{u}|_eAe_{q+i}, \dif\pi_\mu|_{u}\dif L_{u}|_eAe_{q+j}\right) \\
&= g_\mu(uK_\mu)\left(\dif\tau_\mu(u)|_{eK_\mu}\dif\pi_\mu|_e Ae_{q+i}, \dif\tau_\mu(u)|_{eK_\mu}\dif\pi_\mu|_eAe_{q+j}\right) \\
&= \la\dif\pi_\mu|_eAe_{q+i},
 \dif\pi_\mu|_eAe_{q+j}\ra.
\end{align*}

Now we use that $\dif\pi_\mu|_e:\RR^{q+n}\longrightarrow\RR^n$ is the projection relative to $\RR^{q+n}=\RR^q\oplus\RR^n$, the fact that $x=(0,...,0,x_1,...,x_n)$ and (\ref{dexp2}) to get

\begin{align*}
(\tilde{g}_\mu)_{ij}(x) &= \left\la\sum_{k=1}^n\Big(\sum_\alpha b_\alpha^{q+k,q+i}(\mu)x^\alpha\Big)e_{q+k}, \sum_{k=1}^n\Big(\sum_\alpha b_\alpha^{q+k,q+j}(\mu)x^\alpha\Big)e_{q+k}\right\ra \\
&=\sum_{k=1}^n\Big(\sum_{\alpha,\beta} b_\alpha^{q+k,q+i}(\mu) b_\beta^{q+k,q+j}(\mu)x^{\alpha+\beta}\Big).
\end{align*}
If we set
$$
a_\alpha^{ij}(\mu):= \sum_{k=1}^n \sum_{\alpha'+\beta'=\alpha} b_{\alpha'}^{q+k,q+i}(\mu) b_{\beta'}^{q+k,q+j}(\mu)
$$
then $(\tilde{g}_\mu)_{ij}(x)=\sum_\alpha a_\alpha^{ij}(\mu)x^\alpha$, with $\deg(a_\alpha^{ij})=\deg(b_{\alpha'}^{q+k,q+i}b_{\beta'}^{q+k,q+j})=|\alpha'|+|\beta'|=|\alpha|$.  The last assertion on the lower terms easily follows from

\begin{align*}
\left(\frac{I-e^{-\ad_\mu{x}}}{\ad_\mu{x}}\right)_{q+i,q+j} &= \delta_{i,j} -\unm\sum_{k=1}^n\mu_{q+k,q+j}^{q+i}x_k +\tfrac{1}{6}\sum_{k,l=1}^n\Big(\sum_{r=1}^{q+n}\mu_{q+k,r}^{q+i}\mu_{q+l,q+j}^r\Big)\, x_kx_l \notag \\
& + \;\mbox{monomials of degree}\, \geq 3,\notag
\end{align*}
concluding the proof of the proposition.
\end{proof}

\begin{corollary}\label{convmu3}
Let $\mu_k$ be a sequence in $\hca_{q,n}$ such that $\mu_k\to\lambda\in\hca_{q,n}$ and all $\psi_{\mu_k},\psi_\lambda$ are embeddings from an open neighborhood $\Omega$ of $0\in\RR^n$.  Then $\tilde{g}_{\mu_k}\to\tilde{g}_\lambda$ smoothly on $\Omega$.
\end{corollary}

\begin{proof}
The coordinates $(\tilde{g}_\mu)_{ij}$ of the metric $\tilde{g}_\mu$ have been described in Proposition \ref{gmucc} for any $\mu\in\hca_{q,n}$.  We therefore have that
$$
\partial^\beta(\tilde{g}_{\mu_k})_{ij} = \sum_\alpha a_\alpha^{ij}(\mu_k) \partial^\beta x^\alpha,
$$
and since the coefficient $a^{ij}_\alpha(\mu)$ depend polynomially on $\mu_k$, it follows that $a_\alpha^{ij}(\mu_k)\to a_\alpha^{ij}(\lambda)$ uniformly, as $k\to\infty$. This implies that $\tilde{g}_{\mu_k}\to \tilde{g}_\lambda$ smoothly on $\Omega$ (see Remark \ref{convr}), as was to be shown.
\end{proof}

\begin{remark}\label{oexpc}
If instead of canonical coordinates $\psi_\mu=\pi_\mu\circ\exp_{\mu}:B(0,r_\mu)\subset\RR^n\longrightarrow G_\mu/K_\mu$  we use any coordinate system of the form $$
\begin{array}{l}
B(0,r^1_\mu)\times...\times B(0,r^m_\mu)\subset W_1\oplus...\oplus W_m=\RR^n\longrightarrow G_\mu/K_\mu,  \\ \\
(x_1,...,x_m) \mapsto \pi_\mu\left(\exp_\mu(x_1)...\exp_\mu(x_m)\right),
\end{array}
$$
(cf. e.g. \cite[Lemma 2.4]{Hlg}) we can define the corresponding $\tilde{g}_\mu$ and $r^1_\mu,...,r^m_\mu$ will play the role of the Lie injectivity radius for any $\mu\in\hca_{q,n}$ relative to our fixed decomposition $\RR^n=W_1\oplus...\oplus W_m$.  A universal formula for the coordinate $(\tilde{g}_\mu)_{ij}$ analogous to Proposition \ref{gmucc} follows in much the same way, and therefore smooth convergence $\tilde{g}_{\mu_k}\to\tilde{g}_\lambda$ for any convergent sequence $\mu_k\to\lambda$ holds as in Corollary \ref{convmu3}.
\end{remark}

It follows from the proof of Proposition \ref{gmucc} that for any $x$ close to $0$ in $\RR^n$ we have
$$
(\tilde{g}_\mu)_{ij}(x)=\la
\left(I-\unm\ad_\mu{x}+\tfrac{1}{6}(\ad_\mu{x})^2-\dots\right)e_{q+i},
\left(I-\unm\ad_\mu{x}+\tfrac{1}{6}(\ad_\mu{x})^2-\dots\right)e_{q+j}\ra,
$$
where one has to project onto $\RR^n$ before taking the inner product.  It is therefore evident that $\tilde{g}_\mu$ does not depend for instance on $\mu|_{\RR^q\times\RR^q}$, and thus the convergence of a sequence of metrics $\tilde{g}_{\mu_k}\to\tilde{g}_\lambda$ might not affect the brackets completely, in the sense that it might not imply convergence of some part of the brackets $\mu_k$ to the corresponding part of $\lambda$.

We are however in a position to prove that the usual topology on $\hca_{q,n}$ essentially corresponds to infinitesimal convergence (see Definition \ref{infconv}).

\begin{theorem}\label{convmu2}
Let $\mu_k$ be a sequence and $\lambda$ an element in $\hca_{q,n}$.
\begin{itemize}
\item[(i)] If $\mu_k\to\lambda$ in $\hca_{q,n}$ (usual vector space topology), then $(G_{\mu_k}/K_{\mu_k},g_{\mu_k})$ infinitesimally converges to $(G_{\lambda}/K_{\lambda},g_{\lambda})$.

\item[(ii)] If $(G_{\mu_k}/K_{\mu_k},g_{\mu_k})$ infinitesimally converges to $(G_{\lambda}/K_{\lambda},g_{\lambda})$, then
    $$
    \proy_{\RR^n}\circ\mu_k|_{\RR^n\times\RR^n}\to\proy_{\RR^n}\circ\lambda|_{\RR^n\times\RR^n},
    $$
    where $\proy_{\RR^n}:\RR^{q+n}\longrightarrow\RR^n$ is the projection with respect to the decomposition $\RR^{q+n}=\RR^q\oplus\RR^n$.
\end{itemize}
\end{theorem}

\begin{proof}
Let us first prove part (i).  By arguing as in the proof of Corollary \ref{convmu3}, we get that
$\tilde{g}_{\mu_k}\to \tilde{g}_\lambda$ at $0\in\RR^n$, in the sense used in Definition \ref{infconv}.  In other words, $(B(0,r_{\mu_k}),\tilde{g}_{\mu_k})$ infinitesimally converges to $(B(0,r_{\lambda}),\tilde{g}_{\lambda})$, and thus (i) follows. Indeed, if $\Omega_k:=\psi_\lambda(B(0,\tilde{r}))$, where $\tilde{r}:=\min\{r_{\mu_k},r_\lambda\}$, and $\phi_k:=\psi_{\mu_k}\circ\psi_\lambda^{-1}$,
then as $k\to\infty$,
$$
\phi_k^*g_{\mu_k}= (\psi_\lambda^{-1})^*\psi_{\mu_k}^*g_{\mu_k} =(\psi_\lambda^{-1})^*\tilde{g}_{\mu_k}\to (\psi_\lambda^{-1})^*\tilde{g}_{\lambda} =g_\lambda, \quad \mbox{at}\; eK_\lambda.
$$

For part (ii), we first note that if $\nabla^\mu$ denotes the Levi-Civita connection of $\tilde{g}_\mu$, then
$$
\tilde{g}_\mu(0)\left((\nabla^\mu_{e_r}e_j)_0,e_i\right) = \unm(\mu_{rj}^i+\mu_{ri}^j+\mu_{ji}^r), \qquad q+1\leq i,j,r\leq q+n,
$$
(see for instance \cite[7.27]{Bss}), and if $\alpha$ is the multiindex with $1$ at entry $r$ and $0$ elsewhere, then it is easy to see by using Proposition  \ref{gmucc} that
$$
\partial^\alpha(\tilde{g}_\mu)_{ij}(0) = -\unm(\mu_{rj}^i+\mu_{ri}^j), \qquad q+1\leq i,j,r\leq q+n.
$$
Therefore, the convergence $\tilde{g}_{\mu_k}\to \tilde{g}_\lambda$ at $0\in\RR^n$ (recall that this is equivalent to $g_{\mu_k}\to g_\lambda$ at $eK_\lambda$) implies that
$$
(\mu_k)_{rj}^i+(\mu_k)_{ri}^j+(\mu_k)_{ji}^r \to\lambda_{rj}^i+\lambda_{ri}^j+\lambda_{ji}^r, \qquad (\mu_k)_{rj}^i+(\mu_k)_{ri}^j\to\lambda_{rj}^i+\lambda_{ri}^j,
$$
which gives uniform convergence $(\mu_k)_{ji}^r\to\lambda_{ji}^r$ for all $q+1\leq i,j,r\leq q+n$, as $k\to\infty$.  This implies that
$\proy_{\RR^n}\circ\mu_k|_{\RR^n\times\RR^n}\to\proy_{\RR^n}\circ\lambda|_{\RR^n\times\RR^n}$, concluding the proof of the theorem.
\end{proof}

\begin{remark}\label{remconvmu2}
Concerning to what parts of the brackets other than the given in Theorem \ref{convmu2}, (ii) will converge under infinitesimal convergence, we can observe:

\begin{itemize}
\item It follows from the almost-effectiveness condition (h4) (see Section \ref{varhs}) that $\mu_k|_{\RR^q\times\RR^q}$ is determined by $\mu_k|_{\RR^q\times\RR^n}$.

\item If $\mu_k(\RR^n,\RR^n)\subset\RR^n$ for all $k$, then it is easy to prove that infinitesimal convergence is equivalent to only $\mu_k|_{\RR^n\times\RR^n}\to\lambda|_{\RR^n\times\RR^n}$, as $k\to\infty$.  In other words, the isotropy Lie subalgebra and its isotropy representation are not affected at all by the convergence $\tilde{g}_{\mu_k}\to \tilde{g}_\lambda$ at $0\in\RR^n$ if $(\RR^n,\mu_k)$ is a Lie subalgebra (and consequently $G_{\mu_k}$ is a semidirect product).

\item On the other hand, under the assumption $\proy_{\RR^q}\circ\mu_k(\RR^n,\RR^n)=\RR^q$ for all $k$, it is reasonable to expect from Theorem \ref{convmu2}, (ii), the formula for the coefficients of monomials of degree $2$ in the coordinates of $\tilde{g}_{\mu_k}(x)$ (see Proposition \ref{gmucc}) and the first observation above that infinitesimal convergence will imply the full convergence $\mu_k\to\lambda$.
\end{itemize}
\end{remark}

Recall from Example \ref{AW3} that a positive lower bound on the Lie injectivity radii is necessary to get local convergence from brackets convergence.  We now prove that this suffices.

\begin{theorem}\label{convmu}
Let $\mu_k$ be a sequence such that $\mu_k\to\lambda$ in $\hca_{q,n}$, as $k\to\infty$, and assume that $\inf\limits_k r_{\mu_k}>0$.  Then,
\begin{itemize}
\item[(i)] $\tilde{g}_{\mu_k}\to\tilde{g}_\lambda$ smoothly on some ball $B(0,\tilde{r})\subset\RR^n$, $\tilde{r}>0$.

\item[(ii)] $(G_{\mu_k}/K_{\mu_k},g_{\mu_k})$ locally converges to $(G_{\lambda}/K_{\lambda},g_{\lambda})$.

\item[(iii)] There exists a subsequence of $(G_{\mu_k}/K_{\mu_k},g_{\mu_k})$ which converges in the pointed sense to a homogeneous manifold locally isometric to $(G_{\lambda}/K_{\lambda},g_{\lambda})$.

\item[(iv)] $(G_{\mu_k}/K_{\mu_k},g_{\mu_k})$ converges in the pointed sense to $(G_{\lambda}/K_{\lambda},g_{\lambda})$ if $G_\lambda/K_\lambda$ is compact.
\end{itemize}
\end{theorem}

\begin{remark}
Two different subsequences of $(G_{\mu_k}/K_{\mu_k},g_{\mu_k})$ may converge to different limits in the pointed topology if $G_\lambda/K_\lambda$ is not compact (see Example \ref{berger}).
\end{remark}

\begin{remark}\label{oexpc2}
The metrics $\tilde{g}_{\mu_k}$, $\tilde{g}_\lambda$, in part (i) can be replaced by the ones obtained by considering the other possible coordinates described in Remark \ref{oexpc}.  This is often useful as the radii $r^1_\mu,...,r^m_\mu$ may be larger than the Lie injectivity radius $r_\mu$, providing smooth convergence on a larger open subset of $\RR^n$.  We have for example that $\exp:\slg_2(\RR)\longrightarrow\widetilde{\Sl_2}(\RR)$ is not a diffeomorphism, and however $\vp(xe_1+ye_2+ze_3):=\exp(xe_1).\exp(ye_2+ze_3)$ is so if $\{ e_i\}$ is a basis of $\slg_2(\RR)$ such that $[e_2,e_3]=-e_1$, $[e_3,e_1]=e_2$, $[e_1,e_2]=e_3$ (this will be used in Examples \ref{berger} and \ref{hyp} to prove certain pointed convergence).
\end{remark}

\begin{proof}
The first two items follow by arguing as in the proof of Theorem \ref{convmu2} and using that, in this case, we can fix a neighborhood of the form $\Omega=\psi_\lambda(B(0,\tilde{r}))$ of $eK_\lambda$.

From part (ii) and the compactness theorem, we get a subsequence converging to a complete Riemannian manifold $(M,g)$ which is automatically homogeneous.  But such a subsequence also locally converges to $(G_{\lambda}/K_{\lambda},g_{\lambda})$, and so $(M,g)$ must be locally isometric to $(G_{\lambda}/K_{\lambda},g_{\lambda})$.  This proves (iii).  If in addition $G_\lambda/K_\lambda$ is compact, then $M$ is necessarily diffeomorphic to $G_{\lambda}/K_{\lambda}$ as it must be diffeomorphic to $G_{\mu_k}/K_{\mu_k}$ for all $k$ and hence $M$ is simply connected.  As $(M,g)$ is also complete, we get that it is isometric to $(G_{\lambda}/K_{\lambda},g_{\lambda})$ and part (iv) follows.
\end{proof}

We now apply Theorem \ref{convmu} to the following examples.

\begin{example}\label{berger}
For any $\mu=\mu_{a,b,c}$ as in Example \ref{ex0-3} we define
$$
\psi_\mu:\RR\times\RR^2\longrightarrow G_\mu, \quad\psi_\mu(\theta,x,y):=\exp_\mu(\theta e_1).\exp_\mu(xe_2+ye_3).
$$
There exist $r,s >0$ depending on $\mu$ such that $\psi_\mu:(-s,s)\times B(0,r)\longrightarrow G_\mu$ is an embedding.  We know that any convergent sequence $\mu_k\to\lambda$ of these Lie brackets produces a smooth convergence $\psi_{\mu_k}^*g_{\mu_k}\to\psi_\lambda^*g_\lambda$ on any neighborhood of $0\in\RR^3$ where all $\psi_{\mu_k},\psi_\lambda$ be embeddings (see Corollary \ref{convmu3} and Remark \ref{oexpc}).

\no (i) As a first example, we take $\mu_k:=\mu_{-1/k,1,1}\to\mu_{0,1,1}=:\lambda$, as $k\to\infty$, and use that in this case all $\psi_{\mu_k},\psi_\lambda$ are diffeomorphisms from the whole $\RR^3$ to the corresponding Lie group (recall that $G_{\mu_k}\simeq\widetilde{\Sl_2}(\RR)$ for all $k$ and $G_\lambda\simeq E(2)$), to conclude that $(G_{\mu_k},g_{\mu_k})$ converges in the pointed sense to $(G_\lambda,g_\lambda)$, a flat manifold diffeomorphic to $\RR^3$.

\no (ii) Secondly, we consider  $\mu_k:=\mu_{1/k,1,1}\to\mu_{0,1,1}=\lambda$, as $k\to\infty$, a case that is topologically more involved as $G_{\mu_k}\simeq S^3$ for all $k$ and $G_\lambda$ is noncompact.  By using that
$$
h_k:(\RR^3,\mu_1)\longrightarrow (\RR^3,\mu_k=h_k.\mu_1), \qquad h_k=\left[\begin{smallmatrix} 1&&\\ &\sqrt{k}&\\ &&\sqrt{k}\end{smallmatrix}\right],
$$
is an isomorphism of Lie algebras, one easily obtains that
$$
\psi_k:=\psi_{\mu_k}:(-s,s)\times B(0,\sqrt{k}r)\longrightarrow G_{\mu_k}
$$
is an embedding for all $k$, where $r,s>0$ are the existing numbers with this property for $\psi_1$.  If
$$
\RR\times B(0,\sqrt{k}r) \stackrel{p_k}{\longrightarrow}  S^1\times B(0,\sqrt{k}r) \stackrel{\phi_k}{\longrightarrow} G_{\mu_k},
$$
are respectively defined by $p_k(\theta,x,y):=(e^{i\theta/2},x,y)$ and
$$
\phi_k(e^{i\theta},x,y):=\exp_{\mu_k}(2\theta e_1).\exp_{\mu_k}(xe_2+ye_3),
$$
then $\psi_k=\phi_k\circ p_k$  and since $\psi_k$ is an immersion we get that $\phi_k$ is an embedding for all $k$.  As $p_k^*\phi_k^*g_{\mu_k}=\psi_k^*g_{\mu_k}\to\psi_\lambda^*g_\lambda$ smoothly on each open subset of the form $(-s+t,s+t)\times B(0,\sqrt{k}r)$, $t>0$, and $p_k$ is a local isometry, one obtains that $\phi_k^*g_{\mu_k}\to g_\infty$ smoothly on compact subsets of $S^1\times\RR^2$, as $k\to\infty$, where $g_\infty$ is the metric on $S^1\times\RR^2$ defined by $p_\infty^*g_\infty:=\psi_\lambda^*\tilde{g}_\lambda$.  In other words, we conclude that $(G_{\mu_k},g_{\mu_k})$ converges to the flat manifold $S^1\times\RR^2$ in the pointed topology.

\no (iii) We now use the two sequences above to show that the limit for the pointed subconvergence stated in Theorem \ref{convmu} may not be unique.  Indeed, consider the sequence
$$
\mu_k:=\left\{\begin{array}{ll} \mu_{1/k,1,1} & \mbox{if $k$ even}; \\ \\  \mu_{-1/k,1,1} & \mbox{if $k$ odd}; \end{array}\right.
$$
which clearly satisfies $\mu_k\to\lambda=\mu_{0,1,1}$, though we have proved pointed convergence $\mu_{2k}\to S^1\times\RR^2$ and $\mu_{2k+1}\to\RR^3$.

\no (iv)  The sequence $\mu_k:=\mu_{0,k,k}$ diverges as $k\to\infty$; however, they are all flat and diffeomorphic to $\RR^3$ and hence pointed convergence to the euclidean space $\RR^3$ holds.

\no (v) The Ricci eigenvalues of the divergent sequence $\mu_k^{\pm}:=\mu_{\pm 1/\sqrt{k},\sqrt{k},\sqrt{k}}$ satisfy
$$
\{\tfrac{1}{2k},\; \pm 1-\tfrac{1}{2k},\; \pm 1-\tfrac{1}{2k}\}\to \{ 0,\; \pm 1,\; \pm 1\}.
$$
This suggests that some kind of convergence $\mu_k^+\to\RR\times S^2$ or $S^1\times S^2$, and $\mu_k^-\to\RR\times H^2$, should hold.  The first one can not pointed subconverge as the injectivity radii go to $0$, but for the second one, pointed convergence actually holds, and this will be proved in the next example by working on $\hca_{1,3}$ instead of $\hca_{0,3}$ (recall that $\RR\times H^2$ is not reached by the family $\mu_{a,b,c}$).
\end{example}

\begin{example}\label{hyp}
We consider a sequence of the form $\mu_k:=\mu_{a_k,b_k,1,1}\in\hca_{1,3}$ as in Example \ref{ex1-3}, such that $a_1=-1$, $a_k\to 0^-$, $b_1=0$, $b_k\to-1$ and $a_k+b_k\equiv -1$.  Thus the sequence $\mu_k$ consists of left-invariant metrics on $\widetilde{\Sl_2}(\RR)$ (with an extra symmetry) and $\mu_k\to\lambda:=\mu_{0,-1,1,1}$, the manifold $\RR\times H^2$.  In much the same way as Example \ref{berger}, (i), one can construct diffeomorphisms
$$
\begin{array}{l}
\psi_k:\RR^3\longrightarrow G_{\mu_k}/K_{\mu_k}=\left(\RR\times\widetilde{\Sl_2}(\RR)\right)/\RR_k, \\ \\
\psi_\lambda:\RR^3\longrightarrow G_{\lambda}/K_{\lambda}=\RR\times\left(\widetilde{\Sl_2}(\RR)/\RR\right),
\end{array}
$$
such that $\psi_k^*g_{\mu_k}\to\psi_\lambda^* g_\lambda$ smoothly on $\RR^3$.  We therefore obtain that $(G_{\mu_k}/K_{\mu_k},g_{\mu_k})$ converges in the pointed sense to $\RR\times H^2$.  By computing the Ricci eigenvalues, we deduce that for all $k$, $\mu_{-1/\sqrt{k},-1+1/\sqrt{k},1,1}\in\hca_{1,3}$ is isometric to $\mu_k^-\in\hca_{0,3}$ from Example \ref{berger}, (v), which diverges as a sequence of brackets.
\end{example}

\subsection{Lie groups case}\label{lgcase}  Our aim in this section is to go over again the case of left-invariant metrics on Lie groups (i.e. $\hca_{0,n}$), the one which has been mostly applied in the literature (cf. e.g. the survey \cite{cruzchica} and the references there in).  Recall from
Example \ref{ex0-n} that $\hca_{0,n}$ is simply the variety $\lca_n$ of $n$-dimensional Lie algebras, and we identify
$$
\mu\in\lca_n\longleftrightarrow (G_\mu,g_\mu)=(G_\mu,\ip),
$$
where $g_\mu=g_{\ip}\equiv\ip$ denotes the left-invariant metric on the simply connected Lie group $G_\mu$ determined by the fixed inner product $\ip$ we have on the Lie algebra $(\RR^n,\mu)$ of $G_\mu$.  Every $h\in\Gl_n(\RR)$ defines an isometry
$$
(G_{h.\mu},\ip)\longrightarrow(G_\mu,\la h\cdot,h\cdot\ra),
$$
from which we deduce that the orbit $\Gl_n(\RR).\mu\subset\lca_n$ parameterizes the set of all left-invariant metrics on $G_\mu$ and the orbit $\Or(n).\mu$ the subset of those which are equivariantly isometric to $(G_\mu,\ip)$ (notice that $\tilde{U}_\mu=U_\mu=\Gl_n(\RR)$ for any $\mu\in\lca_n$).

The following lower bound for the Lie injectivity radius gives rise to special convergence features for Lie groups which are not valid in the general homogeneous case.  Recall that $\mu\in\lca_n$ is said to be {\it completely solvable} if all the eigenvalues of $\ad_\mu{x}$ are real for any $x$.  In particular, any nilpotent and any Iwasawa-type solvable $\mu$ is completely solvable.

\begin{lemma}\label{lieinjLG}
Let $r_\mu$ be the Lie injectivity radius of $\mu\in\lca_n=\hca_{0,n}$. Then,
\begin{itemize}
\item[(i)] $r_\mu\geq\tfrac{\pi}{\|\mu\|}$.

\item[(ii)] $r_\mu=\infty$ for any completely solvable $\mu$ (in particular, $G_\mu$ is diffeomorphic to $\RR^n$).
\end{itemize}
\end{lemma}

\begin{proof}
It is well known that for any $\mu\in\lca_n$ the neighborhood of $0\in\RR^n$ defined by
$$
V_\mu:=\{ x\in\RR^n: |\Im(c)|<\pi \;\mbox{for any eigenvalue $c$ of}\; \ad_\mu{x}\}
$$
satisfies that $\exp_\mu:V_\mu\longrightarrow G_\mu$ is a diffeomorphism onto its image (see \cite[pp. 112]{Vrd}).  On the other hand, for any eigenvalue $c$ of $\ad_\mu{x}$ one has
$$
|\Im(c)|\leq |c|\leq\left(\tr{\ad_\mu{x}(\ad_\mu{x})^t}\right)^{\unm}\leq \|\mu\|\|x\|,
$$
where $\|\mu\|^2:=\sum\|\mu(e_i,e_j)\|^2=\sum\tr{\ad_\mu{e_i}(\ad_\mu{e_i})^t}$.  This implies that $B(0,\tfrac{\pi}{\|\mu\|})\subset V_\mu$ and so part (i) follows.  Concerning part (ii), it is enough to note that $V_\mu=\RR^n$ in the completely solvable case.
\end{proof}

On the other hand, the parts of the brackets $\mu_k$ which might not be affected by an infinitesimal convergence $\tilde{g}_{\mu_k}\to\tilde{g}_\lambda$ are not present here, as $q=0$.  We can therefore rephrase Theorems \ref{convmu2} and \ref{convmu} in the case of Lie groups in a much stronger way as follows.

\begin{corollary}\label{convmu4}
Let $\mu_k$ be a sequence in $\lca_n=\hca_{0,n}$   Then the following conditions are equivalent:
\begin{itemize}
\item[(i)] $\mu_k\to\lambda$ in $\lca_n$ (usual vector space topology).

\item[(ii)] $(G_{\mu_k},\ip)$ infinitesimally converges to $(G_{\lambda},\ip)$.

\item[(iii)] $(G_{\mu_k},\ip)$ locally converges to $(G_{\lambda},\ip)$.

\item[(iv)] $(G_{\mu_k},\ip)$ converges in the pointed sense to $(G_{\lambda},\ip)$, provided $G_\lambda$ is compact or all $\mu_k$ are completely solvable.

\item[(v)] $g_{\mu_k}\to g_\lambda$ smoothly on $\RR^n$, provided all $\mu_k$ are completely solvable.
\end{itemize}
In any case, there is always a subsequence of $(G_{\mu_k},\ip)$ that is convergent in the pointed sense to a homogeneous manifold locally isometric to $(G_{\lambda},\ip)$.
\end{corollary}

\subsection{Remark on collapsing}\label{collapsing}
The following discussion is in the spirit of \cite[Section 3.11]{Grm}.  Actually much of what has been studied in this paper can be found in Gromov's book \cite{Grm}.

Let $\mu_k$ be a sequence in $\hca_{q,n}$ such that $\mu_k\to\lambda\in V_{q+n}$, and assume that $\lambda\notin\hca_{q,n}$.  Recall from Section \ref{varhs} that this is possible if and only if either (h2) or (h4) fail for $\lambda$, and only if $q>0$, that is, never for left-invariant metrics on Lie groups.

If (h4) does not hold for $\lambda$, then by considering new decompositions of the form
$$
\RR^q= \RR^{q'}\oplus\{ z\in\RR^q:\mu(z,\RR^n)=0\}, \qquad \RR^{q'+n}=\RR^{q'}\oplus\RR^n,
$$
and defining $\lambda'\in V_{q'+n}$ as the restriction of $\lambda$ to $\RR^{q'+n}$ (and projection on if necessary), we obtain that $\lambda'\in\hca_{q',n}$ provided (h2) holds for $\lambda'$.  It is not hard to convince ourselves on the validity of Theorems \ref{convmu2} and \ref{convmu} if we replace $(G_\lambda/K_\lambda,g_\lambda)$ by $(G_{\lambda'}/K_{\lambda'},g_{\lambda'})$ everywhere.

\begin{example}
The sequence $\mu_k:=\mu_{1,1,1,1/k}\in\hca_{1,3}$ from Example \ref{ex1-3} converges to  $\lambda:=\mu_{1,1,1,0}\notin\hca_{1,3}$.  In this case, $\lambda'=\mu_{1,1,1}\in\hca_{0,3}$ as in Example \ref{ex0-3}, a round metric on $S^3$.
\end{example}

The behavior to be understood is therefore under the failure of condition (h2) for $\lambda$.  So that $K_\lambda$ is not closed in $G_\lambda$, and a natural thing to do is to consider its closure $\overline{K_\lambda}$, which is again a connected Lie subgroup of $G_\lambda$ such that $\dim{\overline{K_\lambda}}=q'>q=\dim{K_\lambda}$.  By putting $q'=q+r$, $r>0$ and considering decompositions
$$
\RR^{q+n}=\RR^{q+r}\oplus\RR^{n-r}, \qquad \RR^{q+r}=\RR^{q}\oplus\RR^{r}, \qquad \RR^{n}=\RR^{r}\stackrel{\perp}{\oplus}\RR^{n-r}.
$$
one gets that $\lambda\in\hca_{q+r,n-r}$.  Indeed, both (h1) and (h3) follow easily from the fact that $\Ad_\lambda(\overline{K_\lambda})\subset\Or(n)$, (h2) holds by construction and if (h4) fails, then we can fix it as above and in any case to get $\lambda\in\hca_{q',n-r}$ for some $q'<q+r$.

As $(G_\lambda/K_\lambda,g_\lambda)$ has dimension $n-r<n$, we can just forget about any type of convergence we had studied in this paper as a candidate for
$$
(G_{\mu_k}/K_{\mu_k},g_{\mu_k})\to (G_\lambda/K_\lambda,g_\lambda).
$$
A natural guess is that {\it Gromov-Hausdorff} topology should be involved in some way (cf. e.g. \cite[Chapter 3]{Grm}, \cite[10.1]{Ptr}).  More precisely, we expect pointed Gromov-Hausdorff subconvergence to a homogeneous manifold locally isometric to $(G_\lambda/K_\lambda,g_\lambda)$, and thus we would be in the presence of what is called {\it collapsing with bounded curvature} in the literature (actually with bounded geometry).

\begin{example}\label{coll}
Consider $\mu_k:=\mu_{p_k,1,1,-1,0,1,0,1}\in\hca_{1,5}$ as in Example \ref{ex1-5}, where $p_k\in\QQ$ and $p_k\to\sqrt{2}$ as $k\to\infty$.  Thus $\mu_k$ is a sequence of homogeneous metrics on $S^3\times S^2$ which are pairwise non-equivariantly diffeomorphic and $\mu_k\to\lambda:=\mu_{\sqrt{2},1,1,-1,0,1,0,1}\notin\hca_{1,5}$.  However, if we consider the decomposition
$$
\RR^6=\RR^2\oplus\RR^4 = \la e_0,e_1\ra\oplus\la e_2,...,e_5\ra,
$$
then it is easy to check that $\lambda\in\hca_{2,4}$ and is a product of round metrics on $S^2\times S^2$.  The Ricci eigenvalues of $\mu_k$ are $\{ 1,p_k-\unm,p_k-\unm,\unm,\unm\}$.  In the light of the above speculation, $\mu_k\to\lambda$ would represent a collapsing with bounded geometry from $S^3\times S^2$ to $S^2\times S^2$.
\end{example}

Collapsing of homogeneous manifolds from the algebraic point of view used in this paper will be the object of further study.


\begin{thebibliography}{MMMMM}
\bibitem[AK75]{AlkKml} {\sc D. Alekseevskii, B. Kimel'fel'd},
Structure of homogeneous Riemannian spaces with zero Ricci curvature,
{\it Funktional Anal. i Prilozen} {\bf 9} (1975), 5-11
(English translation: {\it Functional Anal. Appl.} {\bf 9} (1975), 97-102.

\bibitem[AW75]{AlfWll} {\sc S. Aloff, N. Wallach}, An infinite family of distinct $7$-manifolds admitting positively curved Riemannian structures, {\it Bull. AMS} {\bf 81} (1975), 93-97.

\bibitem[AH11]{AndHpp} {\sc B. Andrews, C. Hopper}, The Ricci Flow in Riemannian Geometry, {\it Lect. Notes Math.} (2011), Springer.

\bibitem[B87]{Bss} {\sc A. Besse}, Einstein manifolds, {\it Ergeb. Math.} {\bf 10} (1987), Springer-Verlag,
Berlin-Heidelberg.

\bibitem[BWZ04]{BhmWngZll} {\sc C. B$\ddot{{\rm o}}$hm, M.Y. Wang, W. Ziller}, A variational approach for
compact homogeneous Einstein manifolds, {\it Geom. Funct. Anal.} {\bf 14} (2004),
681-733.

\bibitem[BtD85]{BrcDck} {\sc T. Brocker, T. tom Dieck}, Representations of
compact Lie groups, Springer-Verlag, New York, 1985.

\bibitem[C$^+$07]{libro}  {\sc B. Chow, S.-C. Chu, D. Glickenstein, C.
Guenther, J. Isenberg, T, Ivey, D. Knopf, P. Lu, F. Luo, L. Ni}, The Ricci flow:
Techniques and Applications, Part I: Geometric Aspects, {\it AMS Math. Surv. Mon.}
{\bf 135} (2007), Amer. Math. Soc., Providence.

\bibitem[E08]{Ebr} {\sc P. Eberlein}, Riemannian $2$-step nilmanifolds with
prescribed Ricci tensor, {\it Contemp. Math.} {\bf 469} (2008), 167-195.

\bibitem[GP11]{GlcPyn} {\sc D. Glickenstein, T. Payne}, Ricci flow on three-dimensional, unimodular metric Lie algebras, {\it Comm. Anal. Geom.}, in press.

\bibitem[G81]{Grm} {\sc M. Gromov}, Metric structures for Riemannian and non-Riemannian spaces, {\it Progress in Math.} {\bf 152} (1999), Birkh\"auser.

\bibitem[Gz07]{Gzh} {\sc G. Guzhvina}, The action of the Ricci flow on almost flat manifolds, Ph.D. thesis, Universit\"at M\"unster (2007).

\bibitem[H95]{Hml} {\sc R. Hamilton}, A compactness property for solutions of the Ricci flow, {\it Amer. J. Math.} {\bf 117} (1995), 545-572.

\bibitem[Hb98]{Hbr} {\sc J. Heber}, Noncompact homogeneous Einstein spaces, {\it Invent. math}. {\bf 133} (1998), 279-352.

\bibitem[Hn74]{Hnt} {\sc H. Heintze}, On homogeneous manifolds of negative curvature, {\it Math. Ann.} {\bf 211} (1974), 23–34.

\bibitem[Hl78]{Hlg} {\sc S. Helgason}, Differential geometry, Lie groups and symmetric spaces, {\it GSM} {\bf 34} (1978), AMS.

\bibitem[J11]{Jbl} {\sc M. Jablonski}, Concerning the existence of Einstein and Ricci soliton metrics on solvable Lie groups, {\it Geom. Top.} {\bf 15} (2011), 735-764.

\bibitem[KN69]{KbyNmz} {\sc S. Kobayashi, K. Nomizu}, Foundations of differential geometry, Vol. II, {\it Interscience}  (1969), Wiley, New York.

\bibitem[KS91]{KrcStl} {\sc M. Kreck, S. Stolz}, Some nondiffeomorphic homeomorphic homogeneous $7$-manifolds with positive sectional curvature, {\it J. Diff. Geom.} {\bf 33} (1991), 465-486.

\bibitem[K97]{Krg} {\sc B. Kruggel}, A homotopy classification of certain $7$-manifolds, {\it Trans. AMS} {\bf 349} (1997), 2827-2843.

\bibitem[L01]{soliton} {\sc J. Lauret}, Ricci soliton homogeneous nilmanifolds,
{\it Math. Ann.} \textbf{319} (2001), 715-733.

\bibitem[L06]{minimal}  \bysame, A canonical compatible metric for geometric structures on
nilmanifolds, {\it Ann. Global Anal. Geom.} {\bf 30} (2006), 107-138.

\bibitem[L10a]{standard}  \bysame, Einstein solvmanifolds are standard, {\it Ann. of Math.} {\bf 172} (2010), 1859-1877.

\bibitem[L09]{cruzchica}  \bysame, Einstein solvmanifolds and nilsolitons, {\it Contemp. Math.} {\bf 491} (2009), 1-35.

\bibitem[L10b]{nilricciflow}  \bysame, The Ricci flow for simply connected nilmanifolds, {\it Comm. Anal. Geom.}, in press.

\bibitem[L11]{homRF}  \bysame, Ricci flow of homogeneous manifolds and its solitons, preprint 2011 (arXiv).

\bibitem[LW11]{einsteinsolv}  {\sc J. Lauret, C.E. Will}, Einstein solvmanifolds:
existence and non-existence questions, {\it Math. Annalen} {\bf 350} (2011), 199-225.

\bibitem[M76]{Mln} {\sc J. Milnor}, Curvature of Left-invariant Metrics on Lie
Groups, {\it Adv. Math.} {\bf 21}(1976), 293-329.

\bibitem[N11]{Nkl} {\sc Y. Nikolayevsky}, Einstein solvmanifolds and the
pre-Einstein derivation, {\it Trans. Amer. Math. Soc.} {\bf 363} (2011), 3935-3958.

\bibitem[NT90]{NclTrc}  {\sc L. Nicolodi, F. Tricerri}, On two theorems of I.M. Singer about homogeneous spaces,
{\it Ann. Global Anal. Geom.} \textbf{8} (1990), 193-209.

\bibitem[P10]{Pyn} {\sc T. Payne}, The Ricci flow for nilmanifolds, {\it J. Modern Dyn.} {\bf 4} (2010), 65-90.

\bibitem[Pe98]{Ptr} {\sc P. Petersen}, Riemannian geometry, {\it GTM 171, Springer}  (1998).

\bibitem[PTV96]{PrfTrcVnh}  {\sc F. Pr$\ddot{{\rm u}}$fer, F. Tricerri, L. Vanhecke}, Curvature invariants, differential operators and local homogeneity,
{\it Transactions Amer. Masth. Soc.} \textbf{348} (1996), 4643-4652.

\bibitem[Pu99]{Ptt} {\sc T. P\"uttmann}, Optimal pinching constants of odd dimensional
homogeneous spaces, {\it Invent. math.} {\bf 138} (1999), 631-684.

\bibitem[S60]{Sng} {\sc I.M. Singer}, Infinitesimally homogeneous spaces, {\it Comm. Pure Appl. Math.} {\bf 13} (1960), 685-697.

\bibitem[T06]{Tpp} {\sc P. Topping}, Lectures on the Ricci flow, {\it Lect. Notes London Math. Soc.} {\bf 325} (2006), Cambridge University Press.

\bibitem[V74]{Vrd} {\sc V.S. Varadarajan}, Lie groups, Lie algebras and their representations, Prentice-Hall, Englewood Cliffs, N.J. (1974).

\bibitem[W83]{Wrn} {\sc F. Warner}, Foundations of differentiable manifolds and Lie groups, Springer-Verlag (1983).

\bibitem[Wi10]{Wll} {\sc C.E. Will}, A curve of nilpotent Lie algebras which are not Einstein
nilradicals, {\it Monatsh. Math.} {\bf 159} (2010), 425-437.

\end{thebibliography}
\end{document}